\documentclass[10pt,draft]{amsart}
\usepackage{amsmath,amssymb,amsthm}
\begin{document}
\newtheorem{lem}{Lemma}[section]
\newtheorem{prop}{Proposition}[section]
\newtheorem{cor}{Corollary}[section]
\numberwithin{equation}{section}
\newtheorem{thm}{Theorem}[section]

\theoremstyle{remark}
\newtheorem{example}{Example}[section]
\newtheorem*{ack}{Acknowledgments}

\theoremstyle{definition}
\newtheorem{definition}{Definition}[section]

\theoremstyle{remark}
\newtheorem*{notation}{Notation}
\theoremstyle{remark}
\newtheorem{remark}{Remark}[section]

\newenvironment{Abstract}
{\begin{center}\textbf{\footnotesize{Abstract}}%
\end{center} \begin{quote}\begin{footnotesize}}
{\end{footnotesize}\end{quote}\bigskip}
\newenvironment{nome}

{\begin{center}\textbf{{}}%
\end{center} \begin{quote}\end{quote}\bigskip}

\newcommand{\triple}[1]{{|\!|\!|#1|\!|\!|}}

\newcommand{\xx}{\langle x\rangle}
\newcommand{\ep}{\varepsilon}
\newcommand{\al}{\alpha}
\newcommand{\be}{\beta}
\newcommand{\de}{\partial}
\newcommand{\la}{\lambda}
\newcommand{\La}{\Lambda}
\newcommand{\ga}{\gamma}
\newcommand{\del}{\delta}
\newcommand{\Del}{\Delta}
\newcommand{\sig}{\sigma}
\newcommand{\ome}{\omega}
\newcommand{\Ome}{\Omega}
\newcommand{\C}{{\mathbb C}}
\newcommand{\N}{{\mathbb N}}
\newcommand{\Z}{{\mathbb Z}}
\newcommand{\R}{{\mathbb R}}
\newcommand{\T}{{\mathbb T}}
\newcommand{\Rn}{{\mathbb R}^{n}}
\newcommand{\Rnu}{{\mathbb R}^{n+1}_{+}}
\newcommand{\Cn}{{\mathbb C}^{n}}
\newcommand{\spt}{\,\mathrm{supp}\,}
\newcommand{\Lin}{\mathcal{L}}
\newcommand{\SSS}{\mathcal{S}}
\newcommand{\F}{\mathcal{F}}
\newcommand{\xxi}{\langle\xi\rangle}
\newcommand{\eei}{\langle\eta\rangle}
\newcommand{\xei}{\langle\xi-\eta\rangle}
\newcommand{\yy}{\langle y\rangle}
\newcommand{\dint}{\int\!\!\int}
\newcommand{\hatp}{\widehat\psi}
\renewcommand{\Re}{\;\mathrm{Re}\;}
\renewcommand{\Im}{\;\mathrm{Im}\;}
\def\11{{\rm 1~\hspace{-1.4ex}l} }

\title
[Invariant measures for BO]
{Invariant measures and long time behaviour for the  Benjamin-Ono equation III}
\author[Yu Deng]{Yu~Deng}
\author[Nikolay Tzvetkov]{Nikolay~Tzvetkov}
\author[Nicola Visciglia]{Nicola~Visciglia}
\address{Department of Mathematics, Princeton University, Princeton, NJ 08544},
\email{yudeng@math.princeton.edu}
\address{D\'epartement de Math\'ematiques, Universit\'e de Cergy-Pontoise, 2, 
avenue Adolphe Chauvin, 95302 Cergy-Pontoise  
Cedex, France and Institut Universitaire de France}\email{nikolay.tzvetkov@u-cergy.fr}
\address{Universit\`a Degli Studi di Pisa, Dipartimento di Matematica,
Largo B. Pontecorvo 5 I - 56127 Pisa. Italy}\email{ viscigli@dm.unipi.it}
\maketitle

\begin{abstract}
We complete the program developed in our previous works  aiming to construct an infinite sequence of invariant measures of gaussian type associated with the conservation laws of the Benjamin-Ono equation.
\end{abstract}
\section{Introduction}
Our goal here is to complete the program developed in our previous works   \cite{Deng, tz, TVens, TVimrn, TVjmpa}
aiming to construct an infinite sequence of invariant measures of gaussian type associated with the conservation laws of the Benjamin-Ono equation.

The Benjamin-Ono equation reads 
\begin{equation}\label{Bo}
\partial_t u+ {\mathcal H} \partial_x^2 u + u\partial_x u =0,
\end{equation}
where ${\mathcal H}$ denotes the Hilbert transform. We consider \eqref{Bo} with periodic boundary conditions, i.e. the spatial variable $x$ is on the one dimensional torus.
It is well known (see e.g. \cite{Mats}) that at least formally the solutions of \eqref{Bo} satisfy an infinite number of conservation laws of the form 
\begin{equation}\label{cons-laws}
E_{k/2}(u)= \|u\|_{\dot{H}^{k/2}}^2 + R_{k/2} (u), \quad k=0,1,2,\cdots, 
\end{equation}
where $R_{k/2}$  is a sum of terms homogeneous in $u$ of order larger or equal to three (but contains "less derivatives").
Despite of this remarkable algebraic property, which indicates that the Benjamin-Ono equation is "integrable", there are some 
important analytical difficulties (related to the non local nature of \eqref{Bo}) to develop the inverse scattering method for \eqref{Bo}. In particular we are aware of
no reference implementing integrability methods in the context of the periodic Benjamin-Ono equation.

As already noticed in our previous works, the conservation laws \eqref{cons-laws} may be used to construct invariant measures supported by Sobolev 
spaces of increasing smoothness. This in turn implies some insides on the long time behavior of the solution of \eqref{Bo}. Recall that the idea of using a conserved quantity to construct gaussian type invariant measures goes back to \cite{LRS}. It was then further developed by many authors including 
\cite{B,B2, BB, BB2,BB3, BT2,BTT,Deng0,Deng, NORS,OH,Richards,S,Tz_fourier,TVens,TVimrn,TVjmpa,zh}. 

Let us now briefly recall the construction of gaussian type measures associated with \eqref{cons-laws}.
These measures are absolutely continuous with respect to the gaussian measure $\mu_{k/2}$  induced 
by 
\begin{equation}\label{series}
\varphi_{k/2}(x, \omega)=\sum_{n\in \Z \setminus \{ 0\}} 
\frac{g_n(\omega)}{|n|^{k/2}} e^{{\bf i}nx},
\end{equation}
where, $(g_{n}(\omega))$ is a sequence of centered standard complex gaussian variables such that $g_{n}=
\overline{g_{-n}}$ and  
$(g_{n}(\omega))_{n>0}$ are independent.
For any $N\geq 1$, $k\geq 1$ and $R>0$ we introduce the function
\begin{equation}\label{density}
F_{k/2,N, R}(u)
=\Big(\prod_{j=0}^{k-2} \chi_R (E_{j/2}(\pi_N u)) \Big)
\chi_R (E_{(k-1)/2}(\pi_N u)-\alpha_N) 
e^{-R_{k/2}(\pi_N u)}
\end{equation}
where
$\alpha_N=\sum_{n=1}^N \frac {c}{n}$ for a suitable constant $c$, $\pi_N$ denotes the projector 
on Fourier modes $n$ such that $|n|\leq N$, $\chi_{R}$ is a cut-off function defined as $\chi_R(x)=\chi (x/R)$
with $\chi:\R \rightarrow \R$ a smooth,
compactly supported function such that $\chi(x)=1$ for every $|x|<1$.

It is proved in \cite{tz, TVens} that for every $k\in \N$ with $k\geq 1$ there exists a $\mu_{k/2}$ measurable function 
$F_{k/2,R }(u)$ such that $(F_{k/2,N, R}(u))_{N\geq 1}$ converges to $F_{k/2,R }(u)$ in $L^q(d\mu_{k/2})$ 
for every $1\leq q<\infty$. This 
in particular implies that  $F_{k/2,R }(u)\in L^q(d\mu_{k/2})$. 

Set $d\rho_{k/2,R}\equiv F_{k/2,R }(u)d\mu_{k/2}$.  Then we have that 
$$
\bigcup_{R>0}{\rm supp}(\rho_{k/2,R})={\rm supp}(\mu_{k/2})
$$
and one may conjecture that $\rho_{k/2,R}$ is invariant under a well defined flow of \eqref{Bo}.

This conjecture was proved for $k=1$ in \cite{Deng} and for $k\geq 4$ in \cite{TVimrn,TVjmpa}.
Our goal in this paper is treat the two remaining cases. 
\begin{thm}
The measures $d\rho_{1,R}$ and $d\rho_{3/2,R}$ are
invariant under the flow associated with
the Benjamin-Ono equation \eqref{Bo} established in \cite{M}.
\end{thm}
There are two main sources of difficulties to proof the invariance of  $\rho_{k/2,R}$. 
\\

The first one, presented only for $k\geq 2$, is that even if $E_{k/2}$ is conserved quantity for 
\eqref{Bo} it is no longer conserved by the approximated versions of \eqref{Bo}. This difficulty was resolved for $k\geq 4$ in \cite{TVimrn, TVjmpa} by introducing an argument exploiting in an essential way the random oscillations of the initial data. 
This approach however does not quite see the time oscillations coming from the dispersive
nature of \eqref{Bo}. These time oscillations are quantitively captured by the Bourgain spaces. 
We refer to \cite{NORS} where the time oscillations are used in an essential way in a related context, i.e. to resolve the problem coming from the 
lack of conservation for of the approximation versions of the original problem.  
\\

The second difficulty is the resolution of the Cauchy problem associated with \eqref{Bo} and its approximations on the support of $\rho_{k/2,R}$. For $k\geq 4$ this can be done by standard methods for solving quasilinear hyperbolic PDE's 
(the case $k=4$ is slightly more delicate and already appeals to a dispersive effect). 
For $k=2,3$ the resolution of \eqref{Bo} on the support of 
$\rho_{k/2,R}$ is a delicate issue which is resolved in \cite{M}. The case $k=1$ is even more delicate and was resolved in \cite{Deng} which in turn led to the invariance of 
 $\rho_{1/2,R}$ because in the case $k=1$ the first mentioned difficulty is absent.
\\
 
 In view of the above discussion, we see that the cases $k=2,3$ combine both difficulties and the aim of the present paper is to solve them simultaneously. 
 \\
 
Let us now explain briefly the main novelty in this paper.  
Recall that in the construction of invariant measures, in the infinite dimensional
situation, an important role is played by
a suitable and careful choice of a "good" family of finite dimensional approximating problems. 
In many situations this approximation
can be obtained by projection on the Fourier modes with at most $N$
frequencies, via the sharp Dirichlet projectors $\pi_N$, and letting $N\rightarrow \infty$.
It is however not quite clear whether in the case of the Benjamin-Ono equation
{\it at low level of regularity}  this family of finite dimensional problems
approximate well the true solution of the Benjamin-Ono equation.
To overcome this difficulty we use the idea of \cite{BT2,BTT,Deng} and we project
the equation by using a family of smoothed projectors $S_N^\epsilon$ with $\epsilon>0$ and $N\in \N$ (see the next pargraph). 
However, in contrast with the case treated in \cite{Deng},
where it is sufficient to work with a fixed parameter $\epsilon$ and letting $N\rightarrow \infty$, in the situation we face in this paper, 
it is important to consider both $\epsilon\rightarrow 0$ and $N\rightarrow \infty$, which requires a considerable care. In particular it is of crucial importance that
in Proposition~\ref{deng} below, we have a bound proportional to $t$ which enables to glue local bounds on time intervals with very poor dependence on $\varepsilon$.
In other words, the fact that we do not sacrify any time integration and that we only exploit the random oscillations of the initial data 
in the estimates on the measure evolution is of importance for the analysis in this paper.
\\

We conclude this introduction by fixing some notations.
We denote by ${\mathcal B}(X)$ the Borel sets of the topological space $X$ and by 
$B_M(Y)$ the ball of radius $M$, centered at the origin of a Banach space $Y$.
For every fixed $\epsilon\in (0,1)$ we denote by $\psi_\epsilon$ a smooth function
$\psi_\epsilon:\R\rightarrow \R$ such that
\begin{align}\label{psiepsilon}
\psi_\epsilon(x)=1 \hbox{ for } x\in [0, (1-\epsilon)],  \psi_\epsilon (x)=0 \hbox{ for }
x>1, \\\nonumber\|\psi_\epsilon\|_{L^\infty}=1 \hbox{
and } \psi_{\epsilon}(x)=\psi_\epsilon(|x|).
\end{align}
We denote by $S^\epsilon_N$ the Fourier multiplier:
\begin{equation}\label{SepsilonN}
S^\epsilon_N(\sum_{j\in \Z} a_j e^{ijx})=\sum_{j\in \Z}
a_j \psi_\epsilon(\frac jN) e^{ijx}.
\end{equation}
We also denote by $\Phi(t)$ the flow associated with \eqref{Bo}
(well-defined on $H^s$, $s\geq 0$ thanks to \cite{M}, see also \cite{MP}
for simpler proof)
and by $\Phi_N^\epsilon(t)$ the flow on $H^s$, $s\geq 0$ associated with 
\begin{align}\label{modify}
\partial_t u + {\mathcal H} \partial_x^2 u + S_N^\epsilon (S_N^\epsilon u \cdot S_N^\epsilon u_x)=0.
\end{align}
Since the $x$ mean value is conserved by the flow of \eqref{Bo}, we shall only consider solutions of \eqref{Bo} and of its approximated version \eqref{modify} with vanishing zero Fourier mode (this is the case in \eqref{series} as well).  
\\

{\bf Acknowledgement. }
N.T. is supported by ERC Grant Dispeq,  N.V. is supported by FIRB grant Dinamiche Dispersive.
\section{Deterministic Theory}
In this section we prove the following deterministic result.
\begin{prop}\label{deterministic}
Let $0<\epsilon<1$, $\sigma>\sigma'>0$ and $M>0$ be fixed, so that $\sigma$ is small enough. We have, for some $ T=T(\epsilon,\sigma,\sigma',M)>0$, $C= 
C(\epsilon,\sigma,\sigma',M)>0$ that:
\[\sup_{\phi\in B_{M}(H^{1/2-\sigma'})}\sup_{|t|\leq T}\|\Phi_{N}^{\epsilon}(t)\phi-\Phi(t)\phi\|_{H^{1/2-\sigma}}\leq C N^{-\theta},\] 
where  $\theta=\theta(\sigma,\sigma')>0$.
\end{prop}
\subsection{The spaces}
Let the standard $X^{s,b}$ space be defined by\[\|u\|_{X^{s,b}}^{2}=\sum_{n\in\mathbb{Z}}\int_{\mathbb{R}}\langle n\rangle^{2s}\langle \tau-|n|n\rangle^{2b}|(\mathcal{F}_{x,t}u)(n,\tau)|^{2}\,\mathrm{d}\tau,\] and the $Y^{s}$ space be \[\|u\|_{Y^{s}}^{2}=\sum_{n\in\mathbb{Z}}\langle n\rangle^{2s}\bigg(\int_{\mathbb{R}}|(\mathcal{F}_{x,t}u)(n,\tau)|\,\mathrm{d}\tau\bigg)^{2}.\]We then define the space $U'$ by\[\|u\|_{U'}=\|u\|_{X^{-1/2,12/25}}+\|u\|_{Y^{1/2-\sigma'}},\] and $U$ is defined by replacing $\sigma'$ with $\sigma$. The space $(U')^{T}$ is defined by\[\|u\|_{(U')^{T}}=\sup\{\|v\|_{U'}:v|_{[-T,T]}=u|_{[-T,T]}\},\] while $U^{T}$ and $X^{s,b,T}$ are defined similarly.

In the proof below we will denote \[s=1/2-\sigma, \,\,\,s'=1/2-\sigma',\,\,\,r=1/2+\sigma=1-s.\] The norms we will control in the bootstrap estimate will be $X^{s',r,T}$ and $X^{s,r,T}$ for the gauged function $w$, and $(U')^{T}$ and $U^{T}$ for the original function $u$.
\subsection{Linear bounds}
The content of this section is well-known. We just record it here for the reader's convenience.
In the sequel $\chi\in C^\infty_0(\R)$ is a fixed function such that $\chi(t)=1 \hbox{ for } |t|<1$, $\chi(t)=0 \hbox{ for } |t|>2$.
\begin{prop}\label{linearbound}
We have the following bounds :\\

(1) the Strichartz estimates:
\[\|u\|_{L_{t,x}^{4}}\lesssim \|u\|_{X^{0,3/8}},\] and \[\|u\|_{L_{t,x}^{6}}\lesssim\|u\|_{X^{\sigma,r}}.\]

(2) the bound for the Duhamel evolution: 
 \[\|\mathcal{E}u\|_{X^{s,b}}\lesssim \|u\|_{X^{s,b-1}},\,\,\,\,\,\,0<b<1;\] where the Duhamel operator is defined by
 \[\mathcal{E}u(t)=\chi(t)\int_{0}^{t}\chi(s)e^{(t-s)H\partial_{xx}}u(s)\,\mathrm{d}s.\]

(3) the short-time bound: for $T\leq 1$ and $0<b<b'<1/2$ we have\[\|\chi(T^{-1}t)u\|_{X^{s,b}}\lesssim T^{b'-b}\|u\|_{X^{s,b'}}.\]

(4) fixed-time estimates: 
\[\sup_{t}\|u(t)\|_{H^{1/2-\sigma}}\lesssim\min(\|u\|_{X^{1/2-\sigma},1/2+\sigma},\|u\|_{Y^{1/2-\sigma}}).\]

(5) the linear bounds: 
 \[\|\chi(t)e^{\pm H\partial_{xx}}\phi\|_{X^{s,b}\cap Y^{s}}\lesssim \|\phi\|_{H^{s}}\] for any $s$ and $b$.
\end{prop}
\begin{proof} These are well-known properties of $X^{s,b}$ spaces, see \cite{Tao}.
\end{proof}
\subsection{The gauge transform}\label{transform}
The use of gauge transforms (allowing to weak the impact of the derivative loss) in the context of the Benjamin-Ono equation was initiated by Tao \cite{Tao-BO}.
The computation in this section is a much simplified version of that in \cite{Deng}. We still need the notation from \cite{Deng}, namely that $m_{ij}$ represents the sum $m_{i}+\cdots +m_{j}$ for $i\leq j$.

In this section we are fixing an $\epsilon$ and an $N$; so we will denote $S_{N}^{\epsilon}$ simply by $S$ and $\psi_\epsilon$ by $\psi$. 
Let $u=\Phi_{N}^{\epsilon}(t)(\phi)$ be the global solution to (\ref{modify}) (with initial data $\phi$ of zero mean value). 
Let $z$ be the unique mean zero antiderivative of $u$, and consider the operators $P:g\mapsto(Sz)\cdot g$ and $Q:g\mapsto (Su)\cdot g$. By abusing notation we will also call them $Sz$ and $Su$. The exponential \[M=\exp\bigg(\frac{\mathrm{i}}{2}SPS\bigg)=\sum_{\lambda=0}^{\infty}\frac{1}{\lambda!}(\mathrm{i}/2)^{\lambda}(SPS)^{\lambda}\] is defined as a power series; we will then define \[v=Mu;\,\,\,\,\,\,\,\,w=\pi_{>0}(Mu).\]
The goal is to prove the following
\begin{lem} We have the evolution equation
\begin{equation}\label{neweqn}(\partial_{t}-\mathrm{i}\partial_{xx})w=\mathcal{N}_{2}+\mathcal{N}_{3}, \end{equation}where\[(\mathcal{N}_{2})_{n_{0}}=\sum_{\lambda\geq 0}C_{\lambda}\sum_{n_{0}=n_{1}+n_{2}+m_{1\lambda}}\Lambda_{2}\cdot w_{n_{1}}w_{n_{2}}\prod_{i=1}^{\lambda}\frac{u_{m_{i}}}{m_{i}},\] and \[(\mathcal{N}_{3})_{n_{0}}=\sum_{\lambda\geq 0}C_{\lambda}\sum_{n_{0}=n_{1}+n_{2}+n_{3}+m_{1\lambda}}\Lambda_{3}\cdot u_{n_{1}}u_{n_{2}}u_{n_{3}}\prod_{i=1}^{\lambda}\frac{u_{m_{i}}}{m_{i}}.\] Here $|C_{\lambda}|\lesssim C^{\lambda}/\lambda!$, some of the $w$ may be replaced by $\bar{w}$, and that\[|\Lambda_{2}|\lesssim\min(|n_{0}|,|n_{1}|,|n_{2}|),\,\,\,\,\,\,\,\,|\Lambda_{3}|\lesssim 1,\]
where we use the notation $v_{n}=\hat{v}(n)$. 
\end{lem}
\begin{proof}
The evolution equation satisfied by $w$ can be computed as follows:
\begin{eqnarray}
(\partial_{t}-\mathrm{i}\partial_{xx})w&=&\pi_{>0}M(\partial_{t}-\mathrm{i}\partial_{xx})u+\pi_{>0}[\partial_{t},M]u-\mathrm{i}\pi_{>0}[\partial_{xx},M]u\nonumber\\
&=&-2\mathrm{i}\pi_{>0}(M\pi_{<0}u_{xx})+\pi_{>0}\big([\partial_{t},M]u-\mathrm{i}\big[\partial_{x},[\partial_{x},M]\big]u\big)\nonumber\\
&+&\pi_{>0}(-MS(Su\cdot Su_{x})-2\mathrm{i}[\partial_{x},M]u_{x})\nonumber\\
\label{line1}&=&-2\mathrm{i}\pi_{>0}\partial_{x}(M\pi_{<0}u_{x})\\
\label{line2}&-&2\mathrm{i}\pi_{>0}([\partial_{x},M]-\frac{\mathrm{i}}{2}M(SQS))u_{x}\\
\label{line3}&+&2\mathrm{i}\pi_{>0}[\partial_{x},M]\pi_{<0}u_{x}+\pi_{>0}\big([\partial_{t},M]-\mathrm{i}\big[\partial_{x},[\partial_{x},M]\big]\big)u.
\end{eqnarray}

Expanding $M$ as the power series and writing in Fourier space, we see that the term in (\ref{line1}) has the form $\mathcal{M}_{2}$, where
\[(\mathcal{M}_{2})_{n_{0}}=\sum_{\lambda\geq 0}C_{\lambda}\sum_{n_{1}+m_{1,\lambda}=n_{0},n_{0}>0>n_{1}}(n_{0}n_{1}u_{n_{1}})\cdot \Phi\cdot\prod_{i=1}^{\lambda}\frac{u_{m_{i}}}{m_{i}},\] where $|C_{\lambda}|\leq C^{\lambda}/\lambda!$,  and $\Phi$ is a bounded factor.
Notice that one of $m_{i}$ must be at least $|n_{0}|+|n_{1}|$ in size, we can rearrange the indices and rewrite this as \[(\mathcal{M}_{2})_{n_{0}}=\sum_{\lambda\geq 0}C_{\lambda}\sum_{n_{1}+n_{2}+m_{1,\lambda}=n_{0}}\Lambda\cdot u_{n_{1}}u_{n_{2}}\cdot\prod_{i=1}^{\lambda}\frac{u_{m_{i}}}{m_{i}},\] where $\Lambda$ verifies the bound\[|\Lambda|\leq \min(\langle n_{0}\rangle,\langle n_{1}\rangle,\langle n_{2}\rangle).\]

Next, to analyze the term in (\ref{line2}), notice that $[\partial_{x},SPS]=SQS$, we have that\[[\partial_{x},M]-\frac{\mathrm{i}}{2}M(SQS)=R,\] where\[R=\sum_{\lambda,\mu\geq 0}\frac{1}{(\lambda+\mu+1)!}(\mathrm{i}/2)^{\lambda+\mu+1}(SPS)^{\lambda}[SQS,(SPS)^{\mu}].\] The last commutator can be written as a power of $SPS$, multiplied by\[[SPS,SQS]=SP[S^{2},Q]S+SQ[P,S^{2}]S,\] multiplied by another power of $SPS$. We will only consider the first term, since the second one is similar. First we may commute $\partial_{x}$ with a power of $SPS$ and move it left; since $[\partial_{x},P]=Q$, the error term will be of form $\mathcal{M}_{3}$, where \[(\mathcal{M}_{3})_{n_{0}}=\sum_{\lambda\geq 0}C_{\lambda}\sum_{n_{1}+n_{2}+n_{3}+m_{1,\lambda}=n_{0}}\Lambda\cdot u_{n_{1}}u_{n_{2}}u_{n_{3}}\cdot\prod_{i=1}^{\lambda}\frac{u_{m_{i}}}{m_{i}},\] where $|\Lambda|\lesssim 1$.

Now, let $v=(SPS)^{\mu-1}u$ for some $\mu$, we have \[([S^{2},Q]\partial_{x}v)_{n_{0}}=\mathrm{i}\sum_{n_{0}=n_{1}+m_{1}}n_{1}\big(\psi^{2}(n_{0}/N)-\psi^{2}(n_{1}/N)\big)\psi(m_{1}/N)u_{m_{1}}v_{n_{1}}.\] Plugging in the expression of $v_{n_{1}}$ in terms of $u$, we obtain that \[([S^{2},Q]\partial_{x}v)_{n_{0}}=\mathrm{i}\sum_{n_{0}=n_{2}+m_{1,\mu}}n_{1}\big(\psi^{2}(n_{0}/N)-\psi^{2}(n_{1}/N)\big)\psi(m_{1}/N)u_{m_{1}}u_{n_{2}}\prod_{i=2}^{\mu}\frac{u_{m_{i}}}{m_{i}},\] where $n_{1}=n_{2}+m_{2,\mu}$. Now in this sum, if $\langle m_{i}\rangle\gtrsim\langle n_{1}\rangle$ for some $i\geq 2$, it would be of form $\mathcal{M}_{3}$; otherwise, if $\langle n_{0}\rangle\gtrsim \langle n_{1}\rangle$, it would be of form $\mathcal{M}_{2}$, and if $\langle n_{0}\rangle\ll \langle n_{1}\rangle$, then we have $\langle n_{0}\rangle\ll\langle n_{2}\rangle $ also, so by swapping $n_{2}$ and $m_{1}$ and using symmetry, we find that this term will be of form $\mathcal{M}_{2}$ also.

Next we consider the second term in (\ref{line3}). Recall from Leibniz rule that\[[\partial_{x},M]=\sum_{\lambda,\mu\geq 0}\frac{1}{(\lambda+\mu+1)!}(\mathrm{i}/2)^{\lambda+\mu+1}(SPS)^{\lambda}(SQS)(SPS)^{\mu}.\] If we then commute this with $\partial_{x}$ again and the commutator hits one $SPS$ factor, we will get the same cubic term as $\mathcal{M}_{2}$ above. Therefore, let \[y=\partial_{t}F-\mathrm{i}\partial_{x}u,\] we only need to consider the part \[\sum_{\lambda,\mu\geq 0}\frac{1}{(\lambda+\mu+1)!}(\mathrm{i}/2)^{\lambda+\mu+1}\pi_{>0}(SPS)^{\lambda}(S(Sy)S)(SPS)^{\mu}u.\] We then have from our equation that\[y=-2\mathrm{i}\pi_{<0}u_{x}-\frac{1}{2}\pi_{\neq 0}S(Su)^{2}.\] The second term in the above equation corresponds to a term of form $\mathcal{M}_{3}$; for the first term above, we will combine it with the first term of line (\ref{line3}) to obtain (here we omit the summation in $\lambda$ and $\mu$ which does not affect the estimate anyway)
\[\mathcal{N}=\pi_{>0}[(SPS)^{\lambda}S(Su)S(SPS)^{\mu}(\pi_{<0}u_{x})-(SPS)^{\lambda}S(S\pi_{<0}u_{x})S(SPS)^{\mu}(u)].\] Writing this in Fourier space, we can check that\[(\mathcal{N})_{n_{0}}=\sum_{n_{0}=n_{1}+n_{2}+m_{1\sigma},n_{0}>0>n_{2}}\Lambda\cdot u_{n_{1}}u_{n_{2}}\prod_{i=1}^{\sigma}\frac{u_{m_{i}}}{m_{i}},\] where $\sigma=\lambda+\mu$, $\Lambda$ is nonzero only if all variables are $\lesssim N$, and that\[|\Lambda|\lesssim N^{-1}|n_{2}|\cdot(|n_{0}|+|m_{1}|+\cdots+|m_{\sigma}|).\] Therefore, depending on whether $\max|m_{i}|\gtrsim \min(|n_{0}|,|n_{2}|)$ or not, we can also include this term in either $\mathcal{M}_{3}$ or $\mathcal{M}_{2}$.

Now we need to transform $\mathcal{M}_{2}$ and $\mathcal{M}_{3}$ further into $\mathcal{N}_{2}$ and $\mathcal{N}_{3}$. We will leave $\mathcal{M}_{3}$ as it is, and further consider an $\mathcal{M}_{2}$ term\[\sum_{n_{1}+n_{2}+m_{1,\lambda}=n_{0}}\Lambda\cdot u_{n_{1}}u_{n_{2}}\cdot\prod_{i=1}^{\lambda}\frac{u_{m_{i}}}{m_{i}}.\] Recall that\[u=M^{-1}v=\sum_{\lambda=0}^{\infty}\frac{1}{\lambda!}(-\mathrm{i}/2)^{\lambda}(SPS)^{\lambda}v,\] we have that for positive $n_{1}$,\[u_{n_{1}}=\sum_{\lambda\geq 0}C_{\lambda}\sum_{n_{1}=n_{3}+m_{1\lambda}}\Phi\cdot v_{n_{3}}\prod_{i=1}^{\lambda}\frac{u_{m_{i}}}{m_{i}},\] where $|C_{\lambda}|\lesssim C^{\lambda}/\lambda!$, and $|\Phi|\lesssim1$. Since $u=\bar{u}$, for negative $n_{1}$ we have \[u_{n_{1}}=\sum_{\lambda\geq 0}C_{\lambda}\sum_{n_{1}=n_{3}+m_{1\lambda}}\Phi\cdot (\bar{v})_{n_{3}}\prod_{i=1}^{\lambda}\frac{u_{m_{i}}}{m_{i}}.\] Clearly we may do the same for $n_{2}$. If $n_{1}n_{3}\leq 0$, then there must be some $i$ so that $|m_{i}|\gtrsim |n_{1}|$ (which cancels the $\Lambda$ factor), therefore this counts as a term of $\mathcal{N}_{3}$, upon substituting $v$ by $u$ again. If $n_{1}n_{3}>0$, then we may replace the $v$ on the right hand side by $w$, since we know that $w$ (resp. $\bar{w}$) is supported in the positive (resp. negative) frequencies, so we get $\mathcal{N}_{2}$.

In any case we have reduced each of (\ref{line1}), (\ref{line2}) and (\ref{line3}) to either $\mathcal{N}_{2}$ or $\mathcal{N}_{3}$, this completes the proof.
\end{proof}
\subsection{The bootstrap estimate}
In this section we prove the main a priori estimate, namely the following
\begin{prop}\label{bootstrap}
Let $\epsilon$ and $M$ be fixed. For each $N$, let $u^{N}$ be the solution to (\ref{modify}), with initial data $u^{N}(0)=\phi$, where $\|\phi\|_{H^{s'}}\leq M$. If $N=\infty$ we assume $u^{\infty}$ solves (\ref{Bo}). Moreover, let $w^{N}$ and $w^{\infty}$ be the corresponding gauge transforms. Then, when $T$ is small enough depending on $\epsilon$ and $M$, we can find some functions $\widetilde{u^{N}},\widetilde{w^{N}}$ and $\widetilde{u^{\infty}},\widetilde{w^{\infty}}$ extending $u^{N}$, $u^{\infty}$ and $w^{N}$, $w^{\infty}$ on $[-T,T]$, such that\[\|\widetilde{u}^{N}\|_{U'}+\|\widetilde{u}^{\infty}\|_{U'}+N^{\theta}\|\widetilde{u^{N}}-\widetilde{u^{\infty}}\|_{U}\lesssim_{\epsilon,M}1\] and \[\|\widetilde{w}^{N}\|_{X^{s',r}}+\|\widetilde{w}^{\infty}\|_{X^{s',r}}+N^{\theta}\|\widetilde{w^{N}}-\widetilde{w^{\infty}}\|_{X^{s,r}}\lesssim_{\epsilon,M}1,\] where $\theta>0$ is some constant independent of $N$.
\end{prop}
Notice that\[\sup_{t}\|u(t)\|_{H^{s}}\lesssim\|u\|_{U},\]which follows from proposition \ref{linearbound}, we can see that Proposition \ref{deterministic} is a consequence of Proposition \ref{bootstrap}.
\begin{proof} We only consider the bound for $u^{N}$ (and denote $u^{N}$ by $u$), since the bound for $u^{\infty}$ follows from a similar (and much easier) estimate, and the bound for the difference $u^{N}-u^{\infty}$ follows from a standard procedure of taking differences\footnote[1]{Since we are \emph{not} using any energy estimate which may not be compatible with taking differences.}.

In order to initiate the bootstrap, the first step is to bound the norm $\|u^{N}\|_{U^{T}}$ and $\|w^{N}\|_{X^{s',r,T}}$ for very small $T$. By the standard arguments in $X^{s,b}$ theory, together with part (5) of Proposition \ref{linearbound}, we know that this reduces to proving $\|w(0)\|_{H^{s'}}\lesssim_{M}1$. However, from the expression of the gauge transform we know that\[|w(0)_{n_{0}}|\lesssim\sum_{\mu\geq 0}\frac{C^{\mu}}{\mu!}\sum_{n_{0}=n_{1}+m_{1,\mu}}|u(0)_{n_{1}}|\cdot\prod_{i=1}^{\mu}\frac{|u(0)_{m_{i}}|}{m_{i}}.\] Let the sum over $m_{i}$ be $y_{n_{0}-n_{1}}$, then we have
\[\sum_{l}\langle l\rangle^{3/4}|y_{l}|\lesssim\sum_{m_{1},\cdots,m_{\mu}}\prod_{i=1}^{\mu}\langle m_{i}\rangle^{-1/4}|u(0)_{m_{i}}|\lesssim \|u(0)\|^\mu_{H^{1/4+\sigma}}\lesssim_{M}1,\] and $s'<1/2$, so we can easily deduce that\[\|w(0)\|_{H^{s'}}\lesssim\sum_{l}|y_{l}|\cdot\bigg(\sum_{n}\langle n+l\rangle^{2s'}|u(0)_{n}|^{2}\bigg)^{1/2}\lesssim_{M}1.\]

Suppose we have constructed some $\widetilde{w}$ and $\widetilde{u}$ for some time $T$, satisfying the desired inequalities, we now need to improve these inequalities, with the same $T$, provided that $T\ll_{\epsilon,M}1$. We will first construct a new $\widetilde{w}$, and this is done simply using the equation (\ref{neweqn}). We will define\[w^{*}=\chi(t)e^{\mathrm{i}\partial_{xx}}w(0)+\mathcal{E}(\mathcal{N}_{2}+\mathcal{N}_{3}),\] where $\mathcal{E}$ is the Duhamel operator as in Proposition \ref{linearbound}, and $\mathcal{N}_{2}$ and $\mathcal{N}_{3}$ are constructed \emph{using $\chi(t)\widetilde{w}$ and $\chi(T^{-1}t)\widetilde{u}$ respectively}; however, we will denote these two functions simply by $w$ and $u$ below. Using Proposition \ref{linearbound} again, we now only need to bound\[\|\mathcal{N}_{2}\|_{X^{s',r-1}}+\|\mathcal{N}_{3}\|_{X^{s',r-1}}.\]

To bound $\mathcal{N}_{3}$, we use duality to reduce the bounding the following expression\begin{eqnarray}J&=&\sum_{\mu}C_{\mu}\sum_{n_{0}=n_{13}+m_{1\mu}}\int_{\xi_{0}=\xi_{13}+\eta_{1\mu}}\Lambda\cdot\langle n_{0}\rangle^{s'}(\mathcal{F}_{x,t}v)(n_{0},\xi_{0})\times\nonumber\\
&\times&\prod_{j=1}^{3}(\mathcal{F}_{x,t}u)(n_{j},\xi_{j})\cdot\prod_{i=1}^{\mu}\frac{(\mathcal{F}_{x,t}u)(m_{i},\eta_{i})}{m_{i}}\nonumber,\end{eqnarray} note the abuse of notation by replacing $\chi(T^{-1}t)\widetilde{u}$ with $u$. Here we assume that $v\in X^{0,1/2-\sigma}\subset L_{t,x}^{4}$ (even after taking absolute value in Fourier space), and we may assume without loss of generality that $|n_{0}|\lesssim |n_{1}|$ (the case $|n_{0}|\lesssim |m_{i}|$ is much easier). Now use that $\langle \partial_{x}\rangle^{-s'}u\in L_{t,x}^{2}$, and that $u\in Y^{s'}\subset L_{t,x}^{10}$ when $\sigma$ is small enough, and that $\partial_{x}^{-1}u\in L_{t,x}^{\infty}$ (all hold after taking absolute value in Fourier space), we could simply take absolute value of every term in $J$, then switch to $(t,x)$ space, then use H\"{o}lder to bound $J$. The gain $T^{\theta}$ will come from Proposition \ref{linearbound} and the time cutoff $\chi(T^{-1}t)$ (the same happens below).

Now let us consider the harder part $\mathcal{N}_{2}$. We may omit the summation in $\mu$, and we only need to consider a sum of type
\begin{eqnarray}
J&=&\sum_{n_{0}=n_{1}+n_{2}+m_{1\mu}}\langle n_{0}\rangle^{s'}\langle n_{1}\rangle^{-s'}\langle n_{2}\rangle^{-s'}\min_{0\leq j\leq 2}\langle n_{j}\rangle\times\nonumber\\
&\times&\int_{\xi_{0}=\xi_{1}+\xi_{2}+\eta_{1\mu}+\Xi}F(n_{0},\xi_{0})G(n_{1},\xi_{1})G(n_{2},\xi_{2})\prod_{i=1}^{\mu}\frac{H(m_{i},\eta_{i})}{m_{i}}\nonumber.
\end{eqnarray}Here $\Xi=|n_{0}|n_{0}-|n_{1}|n_{1}-|n_{2}|n_{2}$ and $F$ is defined by\[F(n,\xi)=(\mathcal{F}_{x,t}v)(n,\xi+|n|n)\] with $v$ as above, $G$ and $H$ are defined in the same way, corresponding to functions $\langle\partial_{x}\rangle^{s'}w$ and $u$ respectively. Moreover, we may assume in the summation that $\min\langle n_{j}\rangle\gg\max\langle m_{i}\rangle$, since otherwise we can bound this term in the same way as $\mathcal{N}_{3}$. In this situation we can check algebraically that\[|\Xi|\sim\max_{j}\langle n_{j}\rangle\cdot\min\langle n_{j}\rangle.\] Let $\max\langle n_{j}\rangle =A$ and $\min\langle n_{j}\rangle =B$, then the weight\[\langle n_{0}\rangle^{s'}\langle n_{1}\rangle^{-s'}\langle n_{2}\rangle^{-s'}\min_{0\leq j\leq 2}\langle n_{j}\rangle\lesssim B^{11/20};\] moreover, one of $\xi_{j}$ or $\eta_{i}$ must be $\gtrsim AB$ by our bound on $\Xi$.

Let $|\xi_{j}|\gtrsim AB$ for some $j$, say $j=1$ (the other cases being similar). Notice that $v\in X^{0,1/2-\sigma}\subset L_{t,x}^{4}$, and also $\langle \partial_{x}\rangle^{s'}w\in L_{t,x}^{4}$ (the function that determines $G$), and that we can cancel the weight $B^{11/20}$ by a power $\langle \xi_{1}\rangle^{3/10}$, so we still have\[h\in X^{0,1/5}\subset L_{t,x}^{3}\] by interpolation, where\[(\mathcal{F}_{t,x}h)(n,\xi+|n|n)=\langle\xi\rangle^{3/10}G(n,\xi).\] Now we simply use the above arguments to cancel the weight, then switch to the $(x,t)$ space and ue H\"{o}lder, bounding the $F$ factor in $L_{t,x}^{4}$, one $G$ factor in $L_{t,x}^{4}$ and the other in $L_{t,x}^{3}$, and all $H$ factors in appropriate spaces.

If $|\eta_{i}|\gtrsim AB$ for some $i$ (say $i=1$), then we will use the $X^{-1/2,12/25}$ bound for $u$, which implies\[h\in X^{1/2,1/5}\subset L_{t,x}^{3}\] by H\"{o}lder, where\[(\mathcal{F}_{x,t}h)(n,\xi+|n|n)=|n|^{-1}\langle\xi\rangle^{7/25}H(n,\xi).\] Moreover the $\langle \xi_{1}\rangle^{7/25}$ factor cancels the weight, so we simply bound $F$ and both $G$ factors in $L_{t,x}^{5}$ (using part (1) of Proposition \ref{linearbound} and interpolation), bound the $H$ factor corresponding to $m_{1}$ in $L_{t,x}^{3}$, then bound the other factors in appropriate norms.

Finally we should improve the bound on $u$. We must be careful here, since we will not use the evolution equation of $u$; however, let us postpone this issue to the end, and first see how we can bound the $(U')^{T}$ norm of $u$.

Bounding the $Y^{s'}$ norm is easy; since $u=M^{-1}v$ we can write $u$ as a linear combination of spacetime shifts $(n,\beta)$ of $v$ with coefficients that are summable even after multiplying by $\langle n\rangle^{7/8}$ (this can be proved in the same way as in the analysis of $w(0)$ before), and we know that a spacetime shift $(n,\beta)$ increases the $Y^{s}$ norm by a factor $\lesssim \langle n\rangle^{s'}$.

Now we need to bound the $X^{-1/2,12/25}$ norm of $u$. Clearly we may restrict to $\pi_{>0}u$, so by the formula $u=M^{-1}v$ and duality we only need to bound
\begin{eqnarray}\label{case1}J&=&\sum_{n_{0}=n_{1}+m_{1\mu}}\int_{\beta_{0}=\beta_{1}+\eta_{1\mu}+\Xi}\langle n_{0}\rangle^{-1/2}\langle \beta_{0}\rangle^{12/25}\times\\
&\times&\langle n_{1}\rangle^{-s'}\langle \beta_{1}\rangle^{-r} F(n_{0},\beta_{0})G(n_{1},\beta_{1})\prod_{i=1}^{\mu}\frac{H(m_{i},\eta_{i})}{m_{i}}\nonumber,\end{eqnarray} provided $\langle m_{i}\rangle\ll \langle n_{0}\rangle$ for each $i$, and $H$ is as above, $F$ and $G$ are bounded in $L_{t,x}^{2}$. If instead $\langle m_{i}\rangle\gtrsim\max(\langle n_{0}\rangle,\langle n_{1}\rangle)$ for some $i$, then we should have \begin{eqnarray}\label{case2}J&=&\sum_{n_{0}=n_{1}+m_{1\mu}}\int_{\beta_{0}=\beta_{1}+\eta_{1\mu}+\Xi}\langle n_{0}\rangle^{-1/2}\langle \beta_{0}\rangle^{12/25}\times\\
&\times&\langle n_{1}\rangle^{1/2}\langle \beta_{1}\rangle^{-12/25} F(n_{0},\beta_{0})G(n_{1},\beta_{1})\prod_{i=1}^{\mu}\frac{H(m_{i},\eta_{i})}{m_{i}}\nonumber.\end{eqnarray} In both cases we have \[\Xi=|n_{0}|n_{0}-|n_{1}|n_{1}-|m_{1}|m_{1}-\cdots-|m_{\mu}|m_{\mu}.\] From the equation we know that either $\langle\beta_{1}\rangle\gtrsim\langle\beta_{0}\rangle$, or $\langle\eta_{i}\rangle\gtrsim\langle\beta_{0}\rangle$ for some $i$, or $|\beta_{0}|\lesssim|\Xi|$.

In case (\ref{case1}), if $\langle\beta_{0}\rangle\lesssim\langle\beta_{1}\rangle$, then we can cancel the two powers, then bound $F$ and $G$ in $L_{t,x}^{2}$, the other factors in $L_{t,x}^{\infty}$; if $\langle\beta_{0}\rangle\lesssim\langle\eta_{i}\rangle$ for some $i$, then we simply invoke the $X^{-1/2,12/25}$ bound for $u$ and make similar arguments; if $|\beta_{0}|\lesssim|\Xi|$, notice that \[|\Xi|\lesssim |n_{0}|\cdot\max_{i}|m_{i}|,\] we can use this to cancel the weight, then bound $F$ in $L_{t,x}^{2}$, $G$ in $L_{t,x}^{4}$, the other factors in appropriate spaces.

In case (\ref{case2}), we must have some $i$, so that $|m_{i}|\sim A$ is larger than any other parameter. We may assume in the worst case that $|n_{1}|\sim A$ (since when $|n_{0}|\sim A$ or $|m_{j}|\sim A$ we will gain more due to the powers we have), and the maximum of all other parameters is $B$. Then again we have either $\langle \beta_{0}\rangle\lesssim\langle \beta_{1}\rangle$ or $\langle \beta_{0}\rangle\lesssim\langle \eta_{i}\rangle$ or $|\beta_{0}|\lesssim|\Xi|\lesssim AB$. In the first case we cancel the weight, then bound $F$ and $G$ in $L_{t,x}^{2}$, in the second case, we make similar arguments as before, using the $X^{-1/2,12/25}$ norm of $u$; in the third case we can cancel the weight and gain at least $A^{1/25}$, so the proof still goes through.

Finally let us discuss how to obtain an \emph{improved} estimate without using the evolution equation for $u$. We argue as in \cite{Deng}, first choose some large $K$ depending on the bound $M'$ appearing in the bootstrap assumption, but still smaller than $T^{-1}$; then by decomposing $\pi_{>0}u$ into $\pi_{>K}u$ and $\pi_{[0,K]}u$, we can bound the slightly weaker norm $\|\partial_{x}^{-\sigma/10}u^{*}\|_{U'}$ of some other extension $u^{*}$ of $u$ by $O_{M}(1)$ (in fact, the bound for $\pi_{>K}$ part is trivial since we can gain a power of $K$, and for the $\pi_{[0,K]}$ part we will use the evolution equation for $u$). Then we use the formula $u=M^{-1}v$ and write $v=w+\pi_{\leq 0}v$. We will use $u^{*}$ to realize the operator $M$, use some extension $w^{*}$ of $w$ that is bounded by $O_{M}(1)$ as we just proved. Then we should be able to bound the output function by $O_{M}(1)$ as above, except for the part where we have $\pi_{\leq 0}v$ (which is bounded only by $O_{M'}(1)$ instead of $O_{M}(1)$). But in this case we must have $|m_{i}|\geq K$ for some $i$, so we gain a small power of $K$ which cancels the $O_{M'}(1)$ loss.

In this way we can complete the proof of the proposition.
\end{proof}
\section{Some useful orthogonality relations}\label{orth}
In this section we recall for the sake of completeness some useful results from  \cite{TVjmpa} on the orthogonality of multilinear products of Gaussian variables
$g_k(\omega)$ that appear in \eqref{series}. Introduce the sets : 
$$
{\mathcal A}(n)=\{(j_1,..., j_n)\in \Z^n| j_k\neq 0, k=1,...,n, \sum_{k=1}^n 
j_k=0\},
$$
$$\tilde {\mathcal A}(n)=\{(j_1,..., j_n)\in {\mathcal A}(n)  \,|\,  j_k\neq -j_l, 
\hbox{ } \forall k,l\},
$$
$
\tilde {\mathcal A}^c(n)={\mathcal A}(n)\setminus \tilde {\mathcal A}(n)
$
and
\begin{equation*}
\tilde {\mathcal A}^{c,j}(n)=\{(j_1,...,j_n)\in \tilde {\mathcal A}^c(n)| j=j_l=-j_m \hbox{ for some } 1\leq l\neq m\leq n\}.
\end{equation*}

\begin{prop}\label{orth1}
Assume that 
$$
(j_1,..., j_n), (i_1,...,i_n)\in \tilde {\mathcal A}(n),
\{j_1,...,j_n\}\neq \{i_1,..., i_n\},$$
then 
$
\int g_{j_1}...g_{j_n}
\overline{g_{i_1}...g_{i_n}
} dp=0.
$
\end{prop}
\begin{prop}\label{orth2}
Let $i, j>0$ be fixed and assume \begin{align}\nonumber&(j_1, j_2, j_3, j_4,j_5)\in \tilde {\mathcal A}^{c,j}(5),
(i_1,i_2, i_3, i_4,i_5)\in \tilde {\mathcal A}^{c,i}(5),\\\nonumber
&\big \{\{j_1, j_2, j_3, j_4,j_5\}\setminus \{j, -j\}\big\}\neq \big \{\{i_1, i_2, i_3, i_4,i_5\}\setminus \{i, -i\}\big \},\end{align}
then
$
\int g_{j_1}g_{j_2}g_{j_3}g_{j_4}g_{j_5}
\overline{g_{i_1}
g_{i_2}g_{i_3}
g_{i_4}g_{i_5}
} dp=0$.
\end{prop}
The proof of the propositions above are based on the following lemma.
\begin{lem}
Let
\begin{equation}\label{digraf}
(j_1,...,j_n),(i_1,...,i_n)\in {\mathcal A}(n),
\{j_1,...,j_n\}\neq \{i_1,...,i_n\}\end{equation}
be such that:
\begin{equation}\label{gicie}
\int g_{j_1}...g_{j_n}\overline{g_{i_1}
...g_{i_n}} dp\neq 0.
\end{equation}
Then there exist $1\leq l, m\leq n$, with $l\neq m$ and such that at least one of the following occurs:
either $i_l=-i_m$ or $j_l=-j_m$.
\end{lem}
\begin{proof}
By \eqref{digraf} we get the
existence of $l\in \{1,...n\}$ such that:
\begin{equation}\label{hyHY}|\{k=1,...,n| i_l=i_k\}|\neq 
|\{k=1,...,n| j_k=i_l\}|\end{equation}
where $|.|$ denotes the cardinality.
Next we introduce 
\begin{eqnarray*}
{\mathcal N}_l & = & \{ k=1,...,n \, | \, i_k=-i_l\},
\\
{\mathcal M}_l & = & \{ k=1,...,n \, | \, i_k=i_l\},
\\
{\mathcal P}_l & = & \{ k=1,...,n \, | \, j_k=i_l\},
\\
{\mathcal L}_l & = & \{ k=1,...,n \, | \, j_k=-i_l\}.
\end{eqnarray*}
Notice that ${\mathcal M}_l\neq \emptyset$ since it contains at least the element $l$,
and also by \eqref{hyHY} $|{\mathcal M}_l|\neq |{\mathcal P}_l|$.
We can assume $|{\mathcal M}_l|> |{\mathcal P}_l|$ (the case 
$|{\mathcal M}_l|<|{\mathcal P}_l|$ is similar).
Our aim
is to prove that ${\mathcal N}_l\neq \emptyset$.
Next assume by the absurd that
${\mathcal N}_l= \emptyset$,
then by independence we get
\begin{align}\label{toulou}&\int g_{j_1}...g_{j_n}\overline{g_{i_1}
...g_{i_n}} dp
\\\nonumber=\int |g_{i_l}|^{2|{\mathcal P}_l|}
{\bar g_{i_l}}^{|{\mathcal M}_l| + 
|{\mathcal L}_l|-|{\mathcal P}_l| } dp &\int \big( \Pi_{k\notin {\mathcal M}_l} \overline{g_{i_k}}\big) 
\big(\Pi_{h\notin {\mathcal L}_l\cup {\mathcal P}_l} g_{j_h}\big)
dp=0\end{align}
where at the last step we used $|{\mathcal M}_l| + 
|{\mathcal L}_l|-|{\mathcal P}_l|>0$. Hence we get an absurd by \eqref{gicie}.
\end{proof}
\section{On the approximation of the measures $d\rho_{1,R}$ and $d\rho_{3/2,R}$}
We first introduce the modified energies:
\begin{align}\label{EepsilonN}E_N^\epsilon(u)=  \|u\|_{\dot H^1}^2- 
\|S_N^\epsilon u\|_{\dot H^{1}}^2
+E_{1}(S_N^\epsilon u),
\end{align}
\begin{align}\label{FepsilonN}G_N^\epsilon(u)&= \|u\|_{\dot H^{3/2}}^2 - 
\|S_N^\epsilon u\|_{\dot H^{3/2}}^2
+E_{3/2}(S_N^\epsilon u),
\end{align}
and the approximating modified densities:
\begin{align}\label{old}
F_{N,R}^\epsilon &=\chi_R (\|\pi_N u\|_{L^2})\times 
\chi_R (\|\pi_N u\|_{\dot H^{1/2}}^2 -\alpha_N+ 1/3 \int (S_N^\epsilon u)^3 dx)
\\\nonumber
&\times \exp({
\|S_N^\epsilon u\|_{\dot H^{1}}^2
-E_{1}(S_N^\epsilon u)
}) ,\end{align}
\begin{align}
\label{oldrho} H_{N,R}^\epsilon &= \chi_R (\|\pi_N u\|_{L^2})\times 
\chi_R (\|\pi_N u\|_{\dot H^{1/2}}^2 + 1/3 \int (S_N^\epsilon u)^3 dx) 
\\\nonumber &\times \chi_R(E_N^\epsilon (\pi_N u)-\alpha_N)
\times  \exp({\|S_N^\epsilon u\|_{\dot H^{3/2}}^2
-E_{3/2}(S_N^\epsilon u)}) .
\end{align}
 We recall the explicit expressions of $E_1$ and $E_{3/2}$:
 $$E_1(u)=\|u\|_{\dot H^1}^2+ \frac{3}{4} \int u^2 {\mathcal H} \partial_x u
+ \frac{1}{8} \int  u^4$$
and 
$$E_{3/2}(u)= \|u\|^2_{\dot H^{3/2}} - (\int \frac 32 u u_x^2 +\frac 12 u ({\mathcal H}u_x)^2) \\\nonumber
- \int (\frac 13 u^3 {\mathcal H}u_x +\frac 14 u^2{\mathcal H}(uu_x)) - \frac1{20} \int u^5.$$
Next we prove that as $N\rightarrow \infty$ the measures 
$F_{N,R}^\epsilon d\mu_1$ (for $\epsilon>0$ fixed)
converge to $d\rho_{1,R}$ and $H_{N,R}^\epsilon d\mu_{3/2}$
converge to $d\rho_{3/2,R}$ (in a strong sense). 
\begin{prop}\label{approx}
Let $R, \sigma>0$ and $\epsilon_0>0$ be fixed, then:
\begin{equation}\label{occ}\lim_{N\rightarrow \infty} 
\sup_{A\in {\mathcal B} (H^{1/2-\sigma})} |\int_{A} F_{N,R}^{\epsilon_0} d\mu_1-\int_{A}  d\rho_{1, R}|=0,\end{equation}
\begin{equation}\label{occ3}
\lim_{N\rightarrow \infty} 
\sup_{A\in {\mathcal B} (H^{1-\sigma})} |\int_{A} H_{N,R}^{\epsilon_0} d\mu_1-\int_{A}  d\rho_{3/2, R}|=0.\end{equation}
\end{prop}
The next lemma will be of importance in the sequel.
\begin{lem}\label{unifb}
For every fixed $R>0$, $\epsilon_0>0$, $p\in [1, \infty)$ we have
$$\sup_N \{\big \|F_{N,R}^{\epsilon_0}
\big \|_{L^p(d\mu_1)}, \big \|H_{N,R}^{\epsilon_0}
\big \|_{L^p(d\mu_{3/2})}\}< \infty.$$
\end{lem}
The proof follows modulo minor changes in the argument presented in the analysis in \cite{TVens}.
The only difference is that in this paper we use smoothed projectors $S_N^{\epsilon_0}$
in the definition of the approximating measures,
while in \cite{TVens} we use the sharp projectors $\pi_N$.
This difference however does not affect the argument presented in \cite{TVens}.
\begin{lem}\label{a.e.}
Let $\epsilon_0>1$ be fixed and $\sigma>0$ be small. For every sequence $N_k$ in $\N$, there exists a subsequence $N_{k_h}$ such that:
\begin{equation}\label{iones}F_{N_{k_h}, R}^{\epsilon_0}(u)- F_{N_{k_h},R} (u)\rightarrow 0, \hbox{ a.e. (w.r.t. $d\mu_1$) }
u \in H^{1/2-\sigma},\end{equation}
\begin{equation}\label{iones2}H_{N_{k_h}, R}^{\epsilon_0}(u)- H_{N_{k_h},R} (u)\rightarrow 0, \hbox{ a.e. (w.r.t. $d\mu_{3/2}$) }
u \in H^{1-\sigma},\end{equation}
where
$$
F_{N,R}=\chi_R (\|\pi_N u\|_{L^2})
\chi_R \Big(\|\pi_N u\|_{\dot H^{1/2}}^2 -\alpha_N+ 1/3 \int (\pi_N u)^3 dx\Big) 
e^{-R_{1}(\pi_N u)}
$$
and
$$
H_{N,R}=\chi_R (\|\pi_N u\|_{L^2})
\chi_R (E_{1/2}(\pi_N u)) \chi_R (E_1(\pi_N u) - \alpha_N)  e^{-R_{3/2}(\pi_N u)}
$$
are two of the functions introduced in \eqref{density}.
\end{lem}
\begin{proof}
First we focus on the proof of \eqref{iones}. Notice that if we prove
\begin{equation}\label{first}
\|\int (S_N^{\epsilon_0} u)^2 {\mathcal H} \partial_x (S_N^{\epsilon_0} u) - 
\int (\pi_N u)^2 {\mathcal H} \partial_x (\pi_N u) \|_{L^2(d\mu_1)}
\rightarrow 0 \hbox{ as } N\rightarrow \infty,
\end{equation}
then up to subsequence we get
$$|\int (S_N^{\epsilon_0} u)^2 {\mathcal H} \partial_x (S_N^{\epsilon_0} u) - 
\int (\pi_N u)^2 {\mathcal H} \partial_x (\pi_N u) |\rightarrow 0, \hbox{ a.e. (w.r.t. $d\mu_1$) }
u \in H^{1/2-\sigma}.$$
On the other hand 
\begin{equation}\label{L4}\int (\pi_N u)^4- \int (S_N^{\epsilon_0} u)^4 \rightarrow 0,
\forall u \in H^{1/2-\sigma},\end{equation}
provided that $\sigma>0$ is small enough in such a way that
$H^{{1/2}-\sigma}\subset L^4$.
Hence summarizing we get
\begin{equation}\label{pasa}|R_{N_{k_h}}^{\epsilon_0}(u) - R_{N_{k_h}}(u)|\rightarrow 0 
\hbox{ as } h\rightarrow \infty \hbox{ a.e. (w.r.t. $d\mu_1$) }
u \in H^{1/2-\sigma},\end{equation}
where: $$R_N(u)= 3/4 \int (\pi_N u)^2 {\mathcal H} \partial_x (\pi_N u)+ 1/8\int (\pi_N u)^4$$
and $$R_N^{\epsilon_0}=3/4 \int (S_N^{\epsilon_0} u)^2 {\mathcal H} \partial_x (S_N^{\epsilon_0} u)
+1/8\int (S_N^{\epsilon_0} u)^4.$$
Recall also that following \cite{TVens} one can show that there exists $L$ such
that 
$$
\int (\pi_N u)^2 {\mathcal H} \partial_x (\pi_N u)\rightarrow L
$$ 
in $L^2(d\mu_{1})$,  in particular we have up to
subsequence convergence a.e.
w.r.t. $d\mu_1$ and hence we can assume that up to subsequence
$R_N(u)$ is bounded a.e. w.r.t. $d\mu_1$. By combining this fact with \eqref{pasa}
we deduce:
\begin{equation}\label{fin1}\lim_{k\rightarrow \infty} \exp(-R_{N_{k_h}}^{\epsilon_0}(u))-
\exp(-R_{N_{k_h}}(u))=0,
\hbox{ a.e. (w.r.t. $d\mu_1$) }
u \in H^{1/2-\sigma}.
\end{equation}
On the other hand we have
$$\int (\pi_N u)^3- \int (S_N^{\epsilon_0} u)^3 \rightarrow 0,
\forall
u \in H^{1/2-\sigma},$$
and hence
\begin{equation}\label{fin2}\lim_{N\rightarrow \infty} 
\Big[\chi_R (\|\pi_N u\|_{H^{1/2}}^2 -\alpha_N+ 1/3 \int (S_N^{\epsilon_0} u)^3)
\end{equation}
$$
- \chi_R (\|\pi_N u\|_{H^{1/2}}^2 -\alpha_N+ 1/3 \int (\pi_N u)^3)\Big]
=0 \hbox{ a.e. (w.r.t. $d\mu_1$) }
u \in H^{1/2-\sigma}.$$
We conclude by combining \eqref{fin1}
and \eqref{fin2}.\\
Next we focus on \eqref{first}, whose proof follows by
$$\big\| \sum_{\substack{(j, k,l)\in \Z^3\\ 0<|i| |j|,|k|\leq N\\ j+k+l=0}} \frac {1}{|j||k|}
(1
- \psi_{\epsilon_0} (j/N) \psi_{\epsilon_0} (k/N)\psi_{\epsilon_0} (l/N))g_j g_k 
g_l \big\|_{L^2_\omega}^2\rightarrow 0 \hbox{ as } 
N\rightarrow \infty,$$
that in turn, by an orthogonality argument (as in \cite{TVjmpa}), is equivalent to:
$$
 \sum_{\substack{(j, k,l)\in \Z^3\\ 0<|j|,|k|, |l|\leq N\\ j+k+l=0}} \frac {1}{|j|^2|k|^2}
|1
- \psi_{\epsilon_0} (j/N) \psi_{\epsilon_0} (k/N)\psi_{\epsilon_0} (l/N)|^2\rightarrow 0 \hbox{ as } 
N\rightarrow \infty.$$
Notice that due to the cut--off $\psi_\epsilon$ we can restrict the sum on the set
$$\{(j,k,l)\in \Z^3| j+k+l=0, 0<|j|,|k|, |l|\leq N, \max\{|j|/N, |k|/N, |l|/N\}\geq (1-\epsilon_0)\}
$$
and hence we can control the sum above
by 
$$\sum_{|j|,|k|>(1-\epsilon_0)N/2} \frac1{|j|^2|k|^2}\rightarrow 0 \hbox{ as } 
N\rightarrow \infty.$$
The proof of \eqref{iones2} is similar to the proof of \eqref{iones},
provided that we show:
\begin{equation}\label{firstion}
\|\int (S_N^{\epsilon_0} u) (S_N^{\epsilon_0} u_x)^2 - 
(\pi_N u) (\pi_N u_x)^2\|_{L^2(d\mu_1)}
\rightarrow 0 \hbox{ as } N\rightarrow \infty,
\end{equation}
\begin{equation}\label{firstiontwo}
\|\int (S_N^{\epsilon_0} u) ({\mathcal H}S_N^{\epsilon_0} u_x)^2 - 
(\pi_N u) ({\mathcal H} \pi_N u_x)^2\|_{L^2(d\mu_1)}
\rightarrow 0 \hbox{ as } N\rightarrow \infty,
\end{equation}
and
\begin{align}\label{1ion}|\int (S_N^{\epsilon_0} u)^3 
{\mathcal H}(S_N^{\epsilon_0}u)_x 
- \int (\pi_N u)^3 {\mathcal H}(\pi_N u)_x|\rightarrow 0,
\\\label{2ion}
|\int (S_N^{\epsilon_0} u)^2{\mathcal H}(S_N^{\epsilon_0}uS_N^{\epsilon_0} u_x)
- \int (\pi_N u)^2{\mathcal H}(\pi_N u \pi_N u_x)|\rightarrow 0,
\\\label{3ion}
|\int 
(S_N^{\epsilon_0}u)^5 - \int (\pi_N u)^5|\rightarrow 0,
\\\nonumber
\hbox{ a.e. (w.r.t. $d\mu_1$) }
u \in H^{1-\sigma}\end{align}
The proof of \eqref{3ion} follows by the Sobolev embedding $H^{1-\sigma}\subset L^5$.
To prove \eqref{1ion} (and by a similar argument \eqref{2ion})
we use the following inequality (that follows by fractional integration by parts, 
see page 283 in \cite{TVens}):
$$|\int v_1v_2 v_3{\mathcal H}\partial_x v_4 dx|
\leq C ( \|v_2\|_{L^{\infty}}
\|v_3\|_{L^{\infty}} \|v_1\|_{H^{1/2}}\|v_4\|_{H^{1/2}}+ $$$$
\|v_1\|_{L^{\infty}}
\|v_3\|_{L^{\infty}} \|v_2\|_{H^{1/2}}\|v_4\|_{H^{1/2}} 
+ \|v_1\|_{L^{\infty}}
\|v_2\|_{L^{\infty}} \|v_3\|_{H^{1/2}}\|v_4\|_{H^{1/2}})$$ and hence 
\eqref{1ion} follows
provided that $\|S_N^{\epsilon_0} u-\pi_N u\|_{L^\infty}\rightarrow 0$,
 $\|S_N^{\epsilon_0} u-\pi_N u\|_{H^{1/2}}\rightarrow 0
\hbox{ a.e. (w.r.t. $d\mu_{3/2}$) }
u \in H^{1-\sigma}.$ The second estimate is trivial and the first one follows
since we can select $\rho,p>0$ in such a way that
$W^{\rho,p}\subset L^\infty$ and also $\|v\|_{W^{\rho,p}}<\infty
\hbox{ a.e. (w.r.t. $d\mu_1$) }
u \in H^{1-\sigma}$ (see Proposition 4.2 in \cite{TVens} ). 
The proof of \eqref{firstion} and \eqref{firstiontwo} 
follows the same orthogonality argument as the proof of \eqref{first}.
More precisely we get
\begin{align*}\sum_{\substack{(j,k,l)\in \Z^3, j+k+l=0\\ 0<|j|,|k|, |l|\leq N\\ \max\{|j|/N, |k|/N, |l|/N\}\geq (1-\epsilon_0)}} \frac 1{|j||k||l|^3}
\lesssim \sum_{0<|j|,|k|\leq N, |l|>N(1-\epsilon_0)} \frac 1{|j||k||l|^3}
\\\nonumber + 
\sum_{\substack{j+k+l=0\\0<|j|,|k|\leq N, |j|>N(1-\epsilon_0)}} \frac 1{|j||k||l|^3}
= O(\frac 1{N^\alpha})\end{align*} 
for some $\alpha>0$. In the last estimate we used \cite{tz} (end of page 500).
\end{proof}
\noindent {\em Proof of Proposition \ref{approx}.}
The proof of \eqref{occ} and \eqref{occ3} since now on are the same,
hence we focus on the first one.
It is sufficient to prove that given any sequence $N_k$ in $\N$
there exists a subsequence $N_{k_h}$ such that \eqref{occ} occurs.
Recall that 
$$\lim_{N\rightarrow \infty}
\sup_{A\in {\mathcal B}(H^{1/2-\sigma})}
|\int_{A} F_{N, R} d\mu_1- \int_A d\rho_{1,R}|=0.$$
By combining Lemma \ref{a.e.} with the Egoroff theorem we get 
that, up to subsequence, for every $\epsilon>0$ there exists $\Omega_\epsilon\subset H^{1/2-\sigma}$, with $\sigma>0$,
such that $\mu_1(\Omega_\epsilon)<1-\epsilon$
and
$F^{\epsilon_0}_{N, R}(u) - F_{N, R}(u)\rightarrow 0 \hbox{ in } L^{\infty}(\Omega_\epsilon).$
As a consequence we get
\begin{equation}\label{primhol}|\int_{A\cap \Omega_\epsilon}  F^{\epsilon_0}_{N, R}(u) d\mu_1
- \int_{A\cap \Omega_\epsilon}  F_{N, R}(u) d\mu_1|<\epsilon
\hbox{ for } N>N(\epsilon).\end{equation}
On the other hand by the H\"older inequality
\begin{align}\label{sechol}|\int_{A\cap \Omega_\epsilon^c}  &F^{\epsilon_0}_{N, R}(u) d\mu_1
- \int_{A\cap \Omega_\epsilon^c}  F_{N, R}(u) d\mu_1|\\\nonumber &
\lesssim 
\sup_N \Big(\|  F^{\epsilon_0}_{N, R}\|_{L^2(d\mu_1)} +  \|F_{N, R}\|_{L^2(d\mu_1)}\Big)
\times | \mu_1(\Omega_\epsilon^c) |^{1/2} \lesssim \epsilon^{1/2}\end{align}
where we used Lemma \ref{unifb} and \cite[Proposition 6.5]{TVens}.
The proof follows by combining \eqref{primhol}
with \eqref{sechol}.

\hfill$\Box$

\section{A-priori Gaussian bounds w.r.t. $d\mu_1$}

Recall that for every $\epsilon>0$ we denote by $\psi_\epsilon$ any function
that satisfies \eqref{psiepsilon} and by $S^\epsilon_N$
the associated multiplier defined by \eqref{SepsilonN}.
The main aim of this section is the proof
of the following result.
\begin{prop}\label{spcaart}
Let us denote by $S$ the family of operators 
$S_N^\epsilon$,
for every $N\in \N$, $\epsilon>0$. Then we have: 
\begin{align}\label{1y}
\|\int S \varphi {\mathcal H} (S \varphi_{x}) 
S \varphi S \varphi_x -S\varphi {\mathcal H} (S\varphi_x)  S^2  (S\varphi S\varphi_x)\|_{L^2(d\mu_1(\varphi))}=O(\sqrt \epsilon);
\\\label{2y} 
\|\int  (S\varphi)^3 (S\varphi S\varphi_x)
- (S\varphi)^3 S^2  (S\varphi S\varphi_x)
\|_{L^2(d\mu_1(\varphi))}=O(\epsilon) + O(\frac {\ln N}{\sqrt N}).
\end{align}
\end{prop}
The estimate \eqref{1y}
is equivalent to:
\begin{equation}\label{equiv1y}
\|\sum_{{(a,b,c,d)\in \mathcal A}_N(4)} \Lambda_N^\epsilon (a,b,c,d)
\frac {sign (d)}{|a||b|} g_ag_{b}g_cg_{d}\|_{L^2_\omega}=O(\sqrt \epsilon)
\end{equation}
where $g_e$ are the Gaussian independent variables in \eqref{series},
\begin{equation}\label{Deltamm}\Lambda_N^\epsilon (a,b,c,d)= \psi_\epsilon(\frac{a}N) \psi_\epsilon(\frac{b}N)\psi_\epsilon(\frac{c}N)
\psi_\epsilon(\frac{d}N) [\psi_\epsilon^2(\frac{a+c}N) - 1]
\end{equation}
and
\begin{equation}\label{giglMITmm}{\mathcal A}_N(4)=\{(a,b,c,d)\in {\Z}^4|0<|a|, |b|, |c|, |d|
\leq N, a+b+c+d=0\}.\end{equation}
The proof of \eqref{equiv1y} (and hence \eqref{1y}) is splitted in several lemmas.
\begin{lem}\label{thir}
We have
$\sum_{{(a,b,c,d)\in \mathcal B}_N^\epsilon} \Lambda_N^\epsilon (a,b,c,d)
\frac {sign (d)}{|a||b|} g_ag_{b}g_cg_{d}=0$,
where 
$${\mathcal B}_N^\epsilon=\{(a,b,c,d)\in {\mathcal A}_N(4)| 0<|a|, |b|, |c|, |d| \leq N(1-\epsilon)\}$$
for every $N\in \N$, $\epsilon>0$.
\end{lem}

\begin{proof}
Let us fix $(a,b,c,d)\in {\mathcal B}_N^\epsilon$. 
First we assume $c,d>0$ (the case $c,d<0$ is similar). By the condition $a+b+c+d=0$ we deduce 
that $\min\{a,b\}<0$. Since we are assuming $0<|a|, |b|, |c|, |d|\leq N(1-\epsilon)$,
we get $|a+c|=|b+d|<N(1-\epsilon)$. 
Hence we obtain by the cut--off properties of $\psi_\epsilon$ that $\Lambda_N^\epsilon (a,b,c,d)=0$.\\
Hence we have to consider the case $c\cdot d<0$. Under the extra assumption
$a\neq b$ we deduce that the vectors 
$(a,b,c,d), (a,b,d,c), (b,a, c,d), (b,a,d,c)$ are distinct and belong to ${\mathcal B}_N^\epsilon$. Moreover
$g_ag_bg_cg_d=g_ag_bg_dg_c=g_bg_ag_cg_d=g_bg_ag_dg_c$ and
by simple algebra (recall $a+b+c+d=0$) we get:
$$
\frac{1}{|a||b|} [\Lambda_N^\epsilon (a,b,c,d)sign(d) +\Lambda_N^\epsilon (a,b,d,c)sign(c)$$
$$+\Lambda_N^\epsilon (b,a,c,d)sign(d)+\Lambda_N^\epsilon (b,a,d,c)sign(c)]
=0.$$
The same argument works for $a=b$
(in fact in this case $(a,a,c,d),(a,a,d,c)\in {\mathcal B}_N^\epsilon$ are distinct
since $c$ and $d$ have opposite sign,
$g_ag_ag_cg_d=g_ag_ag_dg_c$ and we have the identity $\frac 1{a^2}\Lambda_N^\epsilon (a,a,c,d)sign(d)
+ \frac 1{a^2} \Lambda_N^\epsilon (a,a,d,c)sign(c)=0$).
 The proof is concluded.
\end{proof}

\begin{lem}\label{thirdf}
We have
$$\sup_N \|\sum_{(a,b,c,d)\in {\mathcal C}_N^\epsilon}
\Lambda_N^\epsilon (a,b,c,d)
\frac {sign (d)}{|a||b|} g_ag_{b}g_cg_{d}\|_{L^2_\omega}=O(\sqrt \epsilon),$$
where 
${\mathcal C}_N^\epsilon=\{(a,b,c,d)\in {\mathcal A}_N(4)|
\max\{|a|, |b|\}>N(1-\epsilon)\}$
and $N\in \N, \epsilon>0$.
\end{lem}

\begin{proof} Assume for simplicity that $|a|>N(1-\epsilon)$
(the case $|b|>N(1-\epsilon)$ is similar). Next we split
${\mathcal C}_N^\epsilon=\tilde {\mathcal C}_N^\epsilon\cup \tilde 
{\mathcal C}_N^{\epsilon,c}$
where 
$$\tilde {\mathcal C}_N^\epsilon=\{(a,b,c,d)\in {\mathcal C}_N^\epsilon| a\neq -b,
a\neq -c,a\neq-d, b\neq -c, b\neq -d, c\neq -d \}$$
and 
$\tilde {\mathcal C}_N^{\epsilon,c}= {\mathcal C}_N^\epsilon\setminus \tilde {\mathcal C}_N^\epsilon$.
Then by orthogonality and Proposition \ref{orth1} we can estimate
$$\|\sum_{(a,b,c,d)\in \tilde {\mathcal C}_N^\epsilon}
\Lambda_N^\epsilon (a,b,c,d)
\frac {(sign (d))}{|a||b|} g_ag_{b}g_cg_{d}\|_{L^2_\omega}^2$$
$$\lesssim (\sum_{N(1-\epsilon)<|a|<N} \frac 1{|a|^2} )\cdot 
(\sum_{0<|b|<N} \frac 1{|b|^2})\cdot ( \sum_{0<|c|<N} 1)
= O(\epsilon).$$
Concerning the sum on the set $\tilde {\mathcal C}_N^{\epsilon,c}$ notice that:
$$\tilde {\mathcal C}_N^{\epsilon,c}=\big (\{(a,-a,a,-a)||a|\neq 0\}\cup \{(a,a,-a,-a)||a|\neq 0\}
\cup
\{(a,-a, b, -b), |a|\neq |b|\}$$$$\cup \{(a, b, -a, -b), |a|\neq |b|\}\cup \{(a,b,-b,-a), |a|\neq |b|\}
\big )\bigcap {\mathcal C}_N^\epsilon.
$$
As a consequence we get  
$(a,b,c,d)\in \tilde {\mathcal C}_N^{\epsilon,c}
$
implies $(-a,-b,-c,-d)\in \tilde {\mathcal C}_N^{\epsilon,c}
$ and moreover $g_ag_bg_cg_d=g_{-a}g_{-b}g_{-c}g_{-d}$.
Since we have the identity $$sign(d)\Lambda_N^\epsilon(a,b,c,d)+sign(-d)\Lambda_N^\epsilon(-a,-b,-c,-d)=0,$$
it is easy to deduce that
$$\sum_{(a,b,c,d)\in \tilde {\mathcal C}_N^{\epsilon,c}}
\Lambda_N^\epsilon (a,b,c,d)
\frac {sign (d)}{|a||b|} g_ag_{b}g_cg_{d}=0.$$

\end{proof}

\begin{lem}\label{thirdf56fght}
We have
$ \sum_{(a,b,c,d)\in {\mathcal D}_N^\epsilon}
\Lambda_N^\epsilon (a,b,c,d)
\frac {sign (d)}{|a||b|} g_ag_{b}g_cg_{d}= 0$,
where
$${\mathcal D}_N^\epsilon=\{(a,b,c,d)\in {\mathcal A}_N(4)| 0<|a|, |b|\leq N(1-\epsilon),
|c|, |d|>N(1-\epsilon)\}
$$ for every $N\in \N$, $\epsilon>0$.
\end{lem}

\begin{proof}
We notice that if $(a,b,c,d)\in {\mathcal D}_N^\epsilon$
then $c\cdot d<0$. In fact assume 
by the absurd that $c, d>0$ or $c,d<0$ then we have
$|c+d|>2N(1-\epsilon)$ and it implies $|a+b|>2N(1-\epsilon)$. This is in contradiction with 
$|a+b|\leq |a|+|b|\leq  2N(1-\epsilon)$.\\
The proof can be concluded arguing as in Lemma~\ref{thir} in the case $c\cdot d<0$.

\end{proof}

\begin{lem}\label{thirdf5678}
We have 
$$\sup_N \|\sum_{(a,b,c,d)\in {\mathcal E}_N^\epsilon}
\Lambda_N^\epsilon (a,b,c,d)
\frac {sign (d)}{|a||b|} g_ag_{b}g_cg_{d}\|_{L^2_\omega}=O(\sqrt \epsilon),$$
where
$${\mathcal E}_N^\epsilon=\{(a,b,c,d)\in {\mathcal A}_N(4)| 0<|a|, |b|, |c|\leq N(1-\epsilon),
|d|>N(1-\epsilon)\}$$$$
\bigcup \{(a,b,c,d)\in {\mathcal A}_N (4)| 0<|a|, |b|, |d|\leq N(1-\epsilon),
|c|>N(1-\epsilon)\},
$$
for every $N\in \N$ and $\epsilon>0$.

\end{lem}

\begin{proof} 
Arguing as in Lemma \ref{thir} in the case $c\cdot d<0$, we can restrict to the case 
$(a,b,c,d)\in {\mathcal E}_N^\epsilon$ with
$c\cdot d>0$.
Next we split the sum on two constraints (see Section~\ref{orth} for the definition
of $\tilde {\mathcal A}(4)$ and $\tilde {\mathcal A}^c(4)$):
$${\mathcal E}_N^\epsilon\cap \tilde {\mathcal A}(4)\cap \{(a,b,c,d)| c\cdot d>0\}$$
and
$${\mathcal E}_N^\epsilon\cap \tilde {\mathcal A}^c(4)\cap \{(a,b,c,d)| c\cdot d>0\}.$$
By combining an orthogonality argument with Proposition~\ref{orth1}
we can estimate the sum on the first constraint by  
$$\big (\sum_{|a+b|>N(1-\epsilon)} \frac1{a^2b^2}\big ) 
\cdot \big (\sum_{N(1-\epsilon)<|d|\leq N} 1\big)
=O(\epsilon),$$
where we used that $c \cdot d>0$ implies
$|a+b|=|c+d|>N(1-\epsilon)$. 
Concerning the sum on the second constraint we have
$${\mathcal E}_N^\epsilon\cap \tilde {\mathcal A}^c(4)\cap \{(a,b,c,d)| c\cdot d>0\}
$$
$$\subset {\mathcal E}_N^\epsilon \bigcap \big (\{(a, b, -a, -b) \} \cup \{(a,b,-b,-a)
\cup \{(a,-a,b,-b)\}
\big )=\emptyset.$$


\end{proof}

\noindent{\em Proof of Proposition \ref{spcaart}.} The proof of 
\eqref{equiv1y} (and hence \eqref{1y}) follows by combining
Lemma \ref{thir}, \ref{thirdf}, \ref{thirdf56fght}, \ref{thirdf5678}.\\
Next we focus on the proof of \eqref{2y}, that can be written
as follows:
\begin{equation}\label{chloe}\|\sum_{(a,b,c,d,e)\in {\mathcal A}_N(5)}
\Gamma^\epsilon_N(a,b,c,d,e)\frac{sign(e)}{|a||b||c||d|}
 g_ag_bg_cg_dg_e
\|_{L^2_\omega}=O(\epsilon) + O(\frac {\ln N}{\sqrt N}),
\end{equation}
where:
\begin{equation}
\label{Gamma}
\Gamma^\epsilon_N(a,b,c,d,e)
=\psi_\epsilon(\frac{a}N) \psi_\epsilon(\frac{b}N)\psi_\epsilon(\frac{c}N)
\psi_\epsilon(\frac{d}N)\psi_\epsilon(\frac{e}N) [1-\psi_\epsilon^2(\frac{d+e}N)],
\end{equation}
and 
\begin{equation}\label{Antilde}{\mathcal A}_N(5)=\{(a,b,c,d,e)\in {\Z}^5|0<|a|, |b|, |c|, |d|, |e|
\leq N, a+b+c+d+e=0\}.\end{equation}
Due to the cut-off properties of $\psi_\epsilon$ the sum
in \eqref{chloe} can be replaced by the sum on the set
\begin{align*}&{\mathcal A}_N(5)\cap \{(a,b,c,d,e)|
|d+e|>N(1-\epsilon)\}\\\nonumber&={\mathcal A}_N(5)\cap \{(a,b,c,d,e)|
|a+b+c|>N(1-\epsilon)\}.\end{align*}
Next we split 
\begin{equation}\label{splitprim}
{\mathcal A}_N(5)= \tilde {\mathcal A}_N(5)\cup \tilde {\mathcal A}_N^{c}(5),
\end{equation}
where: 
\begin{align*}\tilde {\mathcal A}_N(5)&=\{(a,b,c,d,e)\in {\mathcal A}_N(5)|\\\nonumber
&a\notin \{-b,-c, -d, -e\}, b\notin \{-c, -d, -e\}, c\notin \{-d, -e\}, d\neq -e
\}\end{align*}
and 
$\tilde {\mathcal A}_N^{c}(5)={\mathcal A}_N(5)\setminus \tilde {\mathcal A}_N(5)$.
By orthogonality and Proposition \ref{orth1} we get:
\begin{multline*}
\| \sum_{\substack{(a,b,c,d,e)\in \tilde {\mathcal A}_N(5)\\|a+b+c|>N(1-\epsilon)}}
\frac{\Gamma^\epsilon_N(a,b,c,d,e)sign(e)}{|a||b||c||d|}
 g_ag_bg_cg_dg_e
\|_{L^2_\omega}^2
\\
\lesssim
\sum_{\substack{0<|a|,|b|,|c|,|d|\leq N, \\|a+b+c|>N(1-\epsilon)}}
\frac{1}{|a|^2|b|^2|c|^2|d|^2}=O(\frac 1N).
\end{multline*}
Concerning the sum on the set $\tilde {\mathcal A}_N^{c}(5)\cap \{(a,b,c,d,e)|
|a+b+c|>N(1-\epsilon)\} $ we first consider the splitting:
$$\tilde {\mathcal A}_N^{c}(5)=\tilde {\mathcal A}^{c}_{N,a=-b}\cup 
\tilde {\mathcal A}^{c}_{N,a=-c}\cup \tilde {\mathcal A}^{c}_{N,a=-d}\cup \tilde {\mathcal A}^{c}_{N,a=-e}\cup \tilde {\mathcal A}^{c}_{N,b=-c}
$$$$
\cup \tilde {\mathcal A}^{c}_{N,b=-d} \cup \tilde {\mathcal A}^{c}_{N,b=-e}
\cup \tilde {\mathcal A}^{c}_{N,c=-d}\cup \tilde {\mathcal A}^{c}_{N,c=-e}
\cup \tilde {\mathcal A}^{c}_{N,d=-e},
$$
where 
\begin{equation}\label{Aort}
 \tilde {\mathcal A}^{c}_{N,a=-b}= \{(a,b,c,d,e)\in {\mathcal A}_N(5)|a=-b\}
 \end{equation}
and analogous definition for the other sets. First notice that
$$\tilde {\mathcal A}^{c}_{N,d=-e}\cap  \{(a,b,c,d,e)|
|d+e|>N(1-\epsilon)\}=\emptyset.$$ Moreover by orthogonality and Proposition 
\ref{orth2} we can estimate the sum on the set 
$\tilde {\mathcal A}^{c}_{N,a=-b}\cap \{(a,b,c,d,e)|
|a+b+c|>N(1-\epsilon)\}$ (notice that since $a=-b$ we get $|c|>N(1-\epsilon)$) by:
$$\sum_{0<|a|<N} \frac 1{|a|^2} \big( \sum_{\substack{|c|>N(1-\epsilon), \\0<|c|,|d|<N}} \frac{1}{c^2d^2}\big )^\frac 12 =O(\frac 1{\sqrt N}).$$
In a similar way we can treat the sum on 
$\tilde {\mathcal A}^{c}_{N,a=-c}$,
$\tilde {\mathcal A}^{c}_{N,b=-c}$.
Next we treat the sum on 
$\tilde {\mathcal A}^{c}_{N,a=-d}$ (which by symmetry is equivalent to 
$\tilde {\mathcal A}^{c}_{N,b=-d}$, $\tilde {\mathcal A}^{c}_{N,c=-d}$).
In this case we can control the sum by 
$$\sum_{0<|a|\leq N} \frac 1{|a|^2} \big( \sum_{\substack{|a+b+c|>(1-\epsilon)N\\0<|b|,|c|\leq N}} \frac{1}{b^2c^2}\big )^\frac 12
\leq \sum_{0<|a|<N/2(1-\epsilon)} \frac 1{|a|^2} \big( \sum_{\substack{|b+c|>(1-\epsilon)N/2\\0<|b|,|c|\leq N}} \frac{1}{b^2c^2}\big )^\frac 12
$$$$+ \sum_{N/2(1-\epsilon)<|a|<N} \frac 1{|a|^2} \big( \sum_{\substack{0<|b|,|c|\leq N
}} \frac{1}{b^2c^2}\big )^\frac 12=O(\frac 1{\sqrt N}).$$
Next we treat the sum on the set
${\mathcal A}^{c}_{N,a=-e}$ (in the same way we can treat the remaining
cases).
We reduce by orthogonality to the estimate:
\begin{align}\label{butcol}\sum_{0<|e|\leq N} \frac 1{|e|} \big (\sum_{\substack{|e+d|>N(1-\epsilon)\\
0<|d|\leq N}} \frac 1{b^2c^2d^2}\big )^{1/2}
\leq \sum_{0<|e|\leq N} \frac 1{|e|} \big (\sum_{\substack{|e+d|>N(1-\epsilon)\\0<|d|\leq N}} \frac 1{d^2}\big )^{1/2}\\\nonumber=
\sum_{N(1-\epsilon)< |e|\leq N} \frac 1{|e|} \big (\sum_{\substack{|e+d|>
N(1-\epsilon)\\0<|d|\leq N}} \frac 1{d^2}\big )^{1/2}
+ \sum_{0<|e|\leq N(1-\epsilon)} \frac 1{|e|} \big (\sum_{\substack{|e+d|>N(1-\epsilon)\\0<|d|\leq N}} \frac 1{d^2}\big )^{1/2}
\\\nonumber=O(\epsilon) + \sum_{0<|e|\leq N(1-\epsilon)} \frac 1{|e|} \big (\sum_{\substack{|e+d|>N(1-\epsilon)\\0<|d|\leq N}} \frac 1{d^2}\big )^{1/2}.\end{align}
Notice that 
if $|e|\leq N(1-\epsilon)$ and $|e+d|>N(1-\epsilon)$ then either 
$|d|\geq N(1-\epsilon)$ or $e\cdot d>0$. In the first case the sum on the r.h.s. of \eqref{butcol} can be estimated by $O(\frac{\ln N}{\sqrt N})$
and in the second case it can be controlled by:
$$\sum_{0<e\leq N(1-\epsilon)} \frac 1{e} \big (\sum_{\substack{N(1-\epsilon)-e<d\leq N}} \frac 1{d^2}\big )^{1/2}$$
We can suppose that the $e$ summations ranges up to $N(1-\epsilon)-2$ since the contribution of the remaining terms is $O(\frac{1}{ N})$.
Finally,  we can estimate 
$$\sum_{0<e\leq N(1-\epsilon)-2} \frac 1{e} \big (\sum_{\substack{N(1-\epsilon)-e<d\leq N}} \frac 1{d^2}\big )^{1/2}$$
by
\begin{eqnarray*}
& \lesssim & \sum_{0<e<  N(1-\epsilon)-2} \frac 1{e}
\big ( \frac 1{N(1-\epsilon)-e} - \frac 1N \big)^\frac 12
\\
& \lesssim &\sum_{0<e<  N(1-\epsilon)} \frac 1{e}
 \frac{\sqrt {\epsilon N +e}}{\sqrt{ N(N(1-\epsilon)-e)}}
\\
& \lesssim & \sum_{0<e< N(1-\epsilon)} \frac 1{e}
 \frac{1}{\sqrt{ (N(1-\epsilon)-e)}}.
\end{eqnarray*}
By using the variable $\epsilon N+e$ we can rewrite the last sum as follows:
$$...=\sum_{\epsilon N<a<N} \frac1{(a-\epsilon N)\sqrt{N-a}} $$$$
\lesssim \sum_{\epsilon N<a<N/2} \frac1{(a-\epsilon N)\sqrt{N-a}} 
+ \sum_{N/2\leq a<N} \frac1{(a-\epsilon N)\sqrt{N-a}} $$
$$\lesssim \frac 1{\sqrt N} \sum_{\epsilon N<a<N/2} 
\frac1{a-\epsilon N} +\frac 1{N} \sum_{N/2\leq a<N} \frac1{\sqrt{N-a}} = O(\frac {\ln N}{\sqrt N}).$$
It concludes the proof of \eqref{2y}.

\hfill$\Box$
\section{Almost Invariance of $F_{N,R}^\epsilon d\mu_1$}
The main result of this section is the following proposition,
where $F_{N,R}^\epsilon$ is defined in \eqref{old}.
\begin{prop}\label{deng}
Let $\sigma, R >0$ be fixed. Then for every $\delta>0$ there exists
$N=N(\delta)>0$ and $\epsilon=\epsilon(\delta)>0$ such that
$$|\int_A F_{N,R}^\epsilon d\mu_1- \int_{\Phi_{N}^\epsilon(t) A}
F_{N,R}^\epsilon d\mu_1|\leq \delta t$$
for every $A \in
{\mathcal B}(H^{1/2-\sigma})$ and for every $t$.
\end{prop}
\begin{remark}\label{deriv}
Notice that the proposition follows provided that we show
$$\sup_{t, A} |\frac d{dt} \int_{\Phi_{N}^\epsilon(t) A} 
F_{N,R}^\epsilon d\mu_1|\rightarrow 0 \hbox{ as }
\epsilon\rightarrow 0, N\rightarrow \infty.$$
\end{remark}
Using that the modified energies associated with $E_0$ and $E_{1/2}$ are true conservation lows for the approximated flows, we obtain that
the proof of Proposition~\ref{deng} can be completed 
by combining Remark~\ref{deriv} with Proposition 5.4 in \cite{TVimrn}, once we proof the following statement.
\begin{prop}\label{leY}
We have the flowing estimate:
$$
\lim_{\epsilon \rightarrow 0} \Big(\limsup_{N\rightarrow \infty} \big\|\frac d{dt} E^\epsilon_N(\pi_N \Phi_N^\epsilon(t)\varphi)_{t=0}\big\|_{L^2 (d\mu_1(\varphi))}\Big)=0,
$$
where the energies $E^\epsilon_N(u)$ are defined by \eqref{EepsilonN}.
\end{prop}
\begin{proof}
In the sequel we make computations with  
$N\in \N, \epsilon>0$ fixed. Hence for simplicity
we use the following notations:
$S_N^\epsilon=S$, $E_N^\epsilon=\tilde E$.
Notice that if we denote by $u(t)=\pi_N\Phi_N^\epsilon(t) \varphi$ 
the solution to \eqref{modify}, then
$S u(t)$ solves 
\begin{equation}\label{bOM}
Su_t+ {\mathcal H} S u_{xx} + (Su Su_x) +(S^2 - Id) (Su Su_x)=0.
\end{equation}
Next we recall that
\begin{equation}\label{idEfu}\tilde E(u)= E_{1}(Su) + (\|u\|_{\dot H^1}^2 - \|Su\|_{\dot H^1}^2),\end{equation}
where
$$
E_1(u)=\|u\|_{\dot H^1}^2+ 3/4 \int u^2 {\mathcal H} \partial_x u+ 1/8 \int  u^4.
$$
Arguing as in \cite{TVens} and recalling that $Su$ solves 
\eqref{bOM}, then we get:
\begin{align}\label{Ieq}
&\frac d{dt} E_1(Su)_{t=0} =
\\\nonumber  
2 \int S\varphi_x (Id - S^2 ) (S\varphi S\varphi_x)_x
&+3/2 \int (S\varphi) {\mathcal H} (S\varphi_x)  (Id - S^2 ) (S\varphi S\varphi_x)\\
\nonumber +3/4 \int (S\varphi )^2 {\mathcal H} ( (Id - S^2 ) (S\varphi S\varphi_x)_x)
&+ 1/2 \int  (S\varphi)^3(Id - S^2 ) (S\varphi S\varphi_x).
\end{align}
Moreover we have (use that $u(t)$ solves \eqref{modify}):
\begin{align}\label{IIeq}&\frac d{dt} (\|u\|_{\dot H^1}^2 - \|Su\|_{\dot H^1}^2)_{t=0}= \\
2\int (-{\mathcal H} \varphi_{xx} - S(S\varphi S\varphi_x))_x\varphi_x
&-2\int (-{\mathcal H} S\varphi_{xx} - S^2(S\varphi S\varphi_x))_x S\varphi_x
\nonumber \\&= 
2\int S\varphi S\varphi_x S\varphi_{xx} -2 \int S\varphi S\varphi_x S^3\varphi_{xx}
\nonumber \end{align}
(here we used integration by parts and $\int v {\mathcal H} v=0$).
By combining \eqref{idEfu}, 
\eqref{Ieq}, \eqref{IIeq} we get:
\begin{align}\nonumber
&\frac d{dt} \tilde E(\pi_N\Phi_N^\epsilon(t)\varphi)_{t=0}=
3/2 \int (S\varphi) {\mathcal H} (S\varphi_x)  (Id - S^2 ) (S\varphi S\varphi_x)\\
\nonumber &+3/4 \int (S\varphi )^2 {\mathcal H} ( (Id - S^2 ) (S\varphi S\varphi_x)_x)
+ 1/2 \int  (S\varphi)^3(Id - S^2 ) (S\varphi S\varphi_x),
\end{align}
and by integration by parts
\begin{align}\nonumber ...&= 3/2 \int S\varphi {\mathcal H} (S\varphi_x)  (S\varphi S\varphi_x)
-3/2 \int S\varphi {\mathcal H} (S\varphi_x)  S^2  (S\varphi S\varphi_x)
\\\nonumber &
-3/2 \int S\varphi   S\varphi_x {\mathcal H} (S\varphi S\varphi_x)
+3/2 \int S\varphi   S\varphi_x {\mathcal H} S^2 (S\varphi S\varphi_x)
\\\nonumber &+ 1/2 \int  (S\varphi)^3 (S\varphi S\varphi_x)
- 1/2 \int  (S\varphi)^3 S^2  (S\varphi S\varphi_x)
.\end{align}
By using the property $\int v {\mathcal H} v=0$ we deduce
\begin{align}\nonumber ...&=
3/2 \int S\varphi {\mathcal H} (S\varphi_x)  (Id - S^2)  (S\varphi S\varphi_x)
+ 1/2 \int  (S\varphi)^3 (Id - S^2)(S\varphi S\varphi_x)
.\end{align}
We conclude by Proposition~\ref{spcaart}.
\end{proof}
\section{Invariance of $d\rho_{1,R}$}
The proof of the invariance of $d\rho_{1,R}$ follows via standard arguments (see e.g. \cite{TVimrn}) by the following proposition.
\begin{prop}\label{de}
Let $\sigma>0$ small enough and  $\bar t>0$ be fixed. Then for every compact set $A
\subset H^{1/2-\sigma}$ we have: 
\begin{equation}\label{monim}\int_{A} d\rho_{1,R} \leq  \int_{\Phi(\bar t) A} d\rho_{1, R}.
\end{equation}
\end{prop}
\begin{proof} We fix $M>0$ such that $A\subset B_M(H^{1/2-\sigma})
$ and we choose $L>0$ such that
\begin{equation}\label{Lchoice}\Phi(t)(B_M(H^{1/2-\sigma})\subset B_L(H^{1/2-\sigma})
\end{equation}
for every $t\in [0, \bar t]$ (the existence of $L$ follows by \cite{M}).

Next we fix $k>0$ and by Proposition \ref{deng} we get
$N_k\in \N$ and $\epsilon_k>0$ such that:
\begin{equation}\label{1}
|\int_{A} F_{N,R}^{\epsilon_k} d\mu_1 - \int_{\Phi_{N}^{\epsilon_k}(t) A} 
F_{N,R}^{\epsilon_k} d\mu_1|\leq t/k, \quad \quad \forall N>N_k, \quad \forall t.
\end{equation}
On the other hand we have by Proposition \ref{deterministic} the existence of
$t_1=t_1(L,k)>0$ and $C=C(L,k)>0$ such that
$$\sup_{\substack{u\in B_L(H^{1/2-\sigma})\\ t\in [0, t_1]}}
 \|\Phi_{N}^{\epsilon_k}(t)u - \Phi(t)u\|_{H^{1/2-\sigma'}}\leq C N^{-\theta}$$
and hence 
\begin{equation}\label{2sta}\int_{\Phi_{N}^{\epsilon_k}(t) A} d\rho_{1,R} 
 \leq \int_{\Phi(t) A +B_{CN^{-\theta}}(H^{1/2-\sigma'})} 
 d\rho_{1,R}, \hbox{ } \forall t\in [0, t_1].
\end{equation}
In turn by combining \eqref{1} with Proposition \ref{approx} we get the existence of $\tilde N_k\in \N$ such that
\begin{equation}\label{651}|\int_{A} d\rho_{1,R} - \int_{\Phi_{N}^{\epsilon_k}(t) A} 
d\rho_{1,R}|\leq {3t_1}/k \quad \quad \forall N>\tilde N_k, \quad \forall t.
\end{equation}
By combining \eqref{2sta} with \eqref{651} we get 
\begin{equation}\label{2sh}\int_{A} d\rho_{1,R} 
 \leq \int_{\Phi(t) A +B_{CN^{-\theta}}(H^{1/2-\sigma' })} 
 d\rho_{1,R} +3 {t_1}/k, \hbox{ } \forall t\in [0, t_1]
\end{equation}
that by taking the limit as $N\rightarrow \infty$ gives:
$$\int_{A} d\rho_{1,R}  \leq \int_{\Phi(t) A} 
 d\rho_{1,R} +3t_1/k, \hbox{ } \forall t\in [0, t_1].$$
It is sufficient to iterate the bound at most
$[\bar t/t_1]+1$ times and to take the limit as $k\rightarrow \infty$
in order to get \eqref{monim} (notice that we can iterate thanks to 
\eqref{Lchoice}).
\end{proof}
\section{The measures $H_{N, R}^\epsilon d\mu_{3/2}$ and the invariance of $d\rho_{{3/2},R}$}
The proof of the invariance of $d\rho_{3/2, R}$ is similar to the proof of the invariance of $d\rho_{1,R}$, once we establish the following analogue
of Proposition \ref{deng}.
\begin{prop}\label{deng32}
Let $\sigma, R >0$ be fixed. Then for every $\delta>0$ there exists
$N=N(\delta)>0$ and $\epsilon=\epsilon(\delta)>0$ such that
$$|\int_A H_{N,R}^\epsilon d\mu_{3/2}- \int_{\Phi_{N}^\epsilon(t) A}
H_{N,R}^\epsilon d\mu_{3/2}|\leq \delta t$$
for every $A \in
{\mathcal B}(H^{1-\sigma})$ and for every $t$.
\end{prop}
Its proof follows by the following analogue of Proposition \ref{leY}.
\begin{prop}\label{leYu}
We have the following estimate:
$$\lim_{\epsilon\rightarrow 0} \Big(\limsup_{N\rightarrow \infty} 
\big\|\frac d{dt} G^\epsilon_N(\pi_N \Phi_N^\epsilon(t)\varphi)_{t=0}\big\|_{L^2 (d\mu_{3/2}(\varphi))}\Big)=0,
$$
where $G_N^\epsilon(u)$ are defined in \eqref{FepsilonN}.
\end{prop}
In turn we split the proof of this proposition in several steps.
\begin{prop}\label{spca2} We have the following estimates:
\begin{align}\label{321sla}
\|\int S \varphi ({\mathcal H}S\varphi_x) {\mathcal H} (Id - S^2) (S\varphi S\varphi_x)_x\|_{L^2 (d\mu_{3/2}(\varphi))}
=O(\sqrt{\frac{\ln N}{N}})+O(\epsilon),
\\\label{322sla}
\|\int (S\varphi _x)^2 (Id - S^2) (S\varphi S\varphi_x)  \|_{L^2 (d\mu_{3/2}(\varphi))}
=O(\sqrt{\frac{\ln^3 N}{N}})+O(\sqrt \epsilon),
\\\label{323sla}
\|\int ({\mathcal H} S\varphi_x)^2(Id - S^2) (S\varphi S\varphi_x) \|_{L^2 (d\mu_{3/2}(\varphi))}
=O(\sqrt{\frac{\ln^3 N}{N}})+O(\sqrt \epsilon),
\end{align} 
where $S=S^\epsilon_N$.
\end{prop}

We split the proof of \eqref{321sla} in several lemmas.
Notice that we have
$$\|\int S \varphi ({\mathcal H}S\varphi_x) {\mathcal H} (Id - S^2) (S\varphi S\varphi_x)_x\|_{L^2(d\mu_{3/2})}=
$$
$$\|\sum_{{(a,b,c,d)\in \mathcal A}_N(4)} \Delta_N^\epsilon (a,b,c,d)
\frac {|c+d|sign(d)}{|a|^{3/2}|b|^{1/2}|c|^{3/2}|d|^{1/2}} g_ag_{b}g_cg_{d}\|_{L^2_\omega},
$$
where $g_e$ are the Gaussian independent variables in 
\eqref{series},
\begin{equation}\label{Delta} 
\Delta_N^\epsilon (a,b,c,d)=\psi_\epsilon(\frac{a}N) \psi_\epsilon(\frac{b}N)\psi_\epsilon(\frac{c}N)
\psi_\epsilon(\frac{d}N) [1-\psi_\epsilon^2(\frac{c+d}N)]
\end{equation}
and
\begin{equation}\label{giglMIT}
{\mathcal A}_N(4)=\{(a,b,c,d)\in {\Z}^4|0<|a|, |b|, |c|, |d|\leq N, a+b+c+d=0\}.
\end{equation}
\begin{lem}\label{thirsal}
We have
$\sum_{{(a,b,c,d)\in \mathcal B}_N^\epsilon} 
\frac {|c+d|sign(d)}{|a|^{3/2}|b|^{1/2}|c|^{3/2}|d|^{1/2}}\Delta_N^\epsilon (a,b,c,d)
 g_ag_{b}g_cg_{d}=0$
where: 
$${\mathcal B}_N^\epsilon=\{(a,b,c,d)\in {\mathcal A}_N(4)| 0<|a|, |b|, |c|, |d| \leq N(1-\epsilon)\}$$
for any $N\in \N$, $\epsilon>0$.
\end{lem}
\begin{proof}
Notice that in the case  $sign (b)\cdot sign(d)<0$ we have:
$$\frac {|c+d|sign(d)}{|a|^{3/2}|b|^{1/2}|c|^{3/2}|d|^{1/2}}\Delta_N^\epsilon (a,b,c,d)
+ \frac {|a+b|sign(b)}{|a|^{3/2}|b|^{1/2}|c|^{3/2}|d|^{1/2}}\Delta_N^\epsilon (c,d,a,b)=0.$$
Hence we can assume that $sign(b)= sign (d)$.
By the condition $a+b+c+d=0$ we get that 
at least one of the following occurs: either
$sign(b)\neq sign(a)$ or $sign(c)\neq sign(d)$.
In any case we have $|a+b|=|c+d|<N(1-\epsilon)$ and hence we conclude by the cut--off properties of $\psi_\epsilon$ that $\Delta_N^\epsilon (a,b,c,d)=0$.

\end{proof}

\begin{lem}\label{thirdfsal}
We have
$$\|\sum_{(a,b,c,d)\in {\mathcal C}_N^\epsilon}
\frac {|c+d|sign(d)}{|a|^{3/2}|b|^{1/2}|c|^{3/2}|d|^{1/2}}\Delta_N^\epsilon (a,b,c,d)
g_ag_bg_cg_d\|_{L^2_\omega}=O(\sqrt {\frac{\ln N}N})+O(\epsilon),$$
where: 
$${\mathcal C}_N^\epsilon=\{(a,b,c,d)\in {\mathcal A}_N(4)|
\max\{|a|, |c|\}>N(1-\epsilon)\}$$
for every $N\in \N, \epsilon>0$.
\end{lem}

\begin{proof}
We will only treat the case $|a|>N(1-\epsilon)$ (for $|c|>N(1-\epsilon)$
the same argument works). Notice that by the cut--off property of $\psi_\epsilon$ 
we can work on the set ${\mathcal C}_N^\epsilon\cap \{|a|>N(1-\epsilon)\}\cap \{|a+b|=|c+d|>N(1-\epsilon)\}$.\\
We argue as in Lemma
\ref{thirdf} and we split ${\mathcal C}_N^\epsilon=\tilde {\mathcal C}_N^\epsilon\cup \tilde 
{\mathcal C}_N^{\epsilon,c}$.
When we sum on the set $\tilde {\mathcal C}_N^\epsilon$ then we can combine 
an orthogonality argument with Proposition \ref{orth1} and 
with the identity $d=-a-b-c$, and we are reduced to:
\begin{align*}&\sum_{\substack{|a|>N(1-\epsilon)\\ 0<|b|, |c|\leq N}} \frac{1}{|a|^3|b||c|} +
\sum_{\substack{|a|>N(1-\epsilon)\\ 0<|b|, |c|\leq N}} \frac{|a+b+c|}{|a|^3|b||c|^3}
\leq  \sum_{\substack{|a|>N(1-\epsilon)\\ 0<|b|, |c|\leq N}} \frac{1}{|a|^3|b||c|} \\\nonumber
&+
\sum_{\substack{|a|>N(1-\epsilon)\\ 0<|b|, |c|\leq N}} \frac{1}{|a|^2|b||c|^3}
+ \sum_{\substack{|a|>N(1-\epsilon)\\ 0<|b|, |c|\leq N}} \frac{1}{|a|^3|c|^3}
+\sum_{\substack{|a|>N(1-\epsilon)\\ 0<|b|, |c|\leq N}} \frac{1}{|a|^3|b||c|^2}
\\\nonumber&=O(\frac{\ln N}{N}).\end{align*}
Concerning the sum on ${\mathcal C}_N^{\epsilon,c}$
we work on the set:
$$\{(a,a,-a,-a)||a|\neq 0\}
\cup \{(a, b, -a, -b), |a|\neq |b|\}\cup \{(a,b,-b,-a), |a|\neq |b|\}
\big ).$$
Notice that
$$\frac {|a+b|sign(-b)}{|a|^{3/2}|b|^{1/2}|a|^{3/2}|b|^{1/2}}\Delta_N^\epsilon (a,b,-a,-b)
+ \frac {|a+b|sign(b)}{|a|^{3/2}|b|^{1/2}|c|^{3/2}|d|^{1/2}}\Delta_N^\epsilon (-a,-b,a,b)=0$$
and hence we have to consider the sum on the set:
$$\{(a,a,-a,-a)||a|\neq 0\}
\cup \{(a,b,-b,-a), |a|\neq |b|\}.
$$
We can conclude by combining the Minkowski inequality
with the following estimates:
$$\sum_{2|a|>N(1-\epsilon)} \frac{1}{|a|^3} + \sum_{\substack{0<|a|,|b|\leq N\\
|a|>N(1-\epsilon)\\|a+b|>N(1-\epsilon)}}
\frac{|a+b|}{|a|^2|b|^2}$$
$$\leq O(\frac 1{N^2}) + \sum_{\substack{
0<|b|\leq N\\
N(1-\epsilon)<|a|\leq N}}
\frac{1}{|a||b|^2}+\sum_{\substack{
0<|b|\leq N\\
N(1-\epsilon)<|a|\leq N}}
\frac{1}{|a|^2|b|} =O(\frac {\ln N}{N}) + O(\epsilon).$$

\end{proof}

\begin{lem}\label{thirdf56fghtsal}
We have
$$ \sum_{(a,b,c,d)\in {\mathcal D}_N^\epsilon} \frac {|c+d|sign(d)}{|a|^{3/2}|b|^{1/2}|c|^{3/2}|d|^{1/2}}
\Delta_N^\epsilon (a,b,c,d)
 g_ag_{b}g_cg_{d}=0,$$
where:
$${\mathcal D}_N^\epsilon=\{(a,b,c,d)\in {\mathcal A}_N(4)| 0<|a|, |c|\leq N(1-\epsilon),
\max\{|b|, |d|\}>N(1-\epsilon)\}
$$ for every $N\in \N$, $\epsilon>0$.
\end{lem}

\begin{proof}
Arguing as in Lemma \ref{thirsal} we can assume that $sign(b)=sign(d)$, and also
by the cut-off property of $\psi_{\epsilon}$ we can restrict to the set
$|a+b|=|c+d|>N(1-\epsilon)$.
We claim that ${\mathcal D}_N^\epsilon\cap \{sign(b)=sign(d)\}
\cap \{|a+b|>N(1-\epsilon)\}=\emptyset$,
and it will conclude the proof.
We can assume $b, d>0$ (the case $b,d<0$ can be treated in a similar way).
This implies that $0<\min\{b, d\}\leq N(1-\epsilon)$. In fact in case it is not true then we get
$b+d>2N(1-\epsilon)$ which implies
$|a+c|=b+d>2N(1-\epsilon)$, and it is in contradiction with $0<|a|, |c|\leq N(1-\epsilon)$.
We assume for simplicity $0<b\leq N(1-\epsilon)$.
Moreover by the condition
$|a+b|>N(1-\epsilon)$ we get (since $|a|<N(1-\epsilon)$
and $0<b\leq N(1-\epsilon)$) $a>0$.
By combining $a+b+c+d=0$ and $a, b, d>0$, we get $c<0$.
Hence we have
$|c|=a+b+d$ and it is absurd since $0<|c|\leq N(1-\epsilon)$ and
$a+b+d>d\geq N(1-\epsilon)$.


\end{proof}

\noindent{\em Proof of Proposition \ref{spca2}.}
The proof of \eqref{321sla} follows by combining Lemma \ref{thirsal},
\ref{thirdfsal} and \ref{thirdf56fghtsal}.
Concerning the proof of \eqref{322sla}
first notice that
\begin{align}\label{sumps}\|\int (S\varphi_x)^2 (Id - S^2)(S\varphi S\varphi_x)\|_{L^2(d\mu_{3/2})}^2
\\\nonumber=\|\sum_{{(a,b,c,d)\in \mathcal A}_N(4)} \Delta_N^\epsilon (a,b,c,d)
\frac {sign(a)sign(b)sign(d)}{|a|^{1/2}|b|^{1/2}|c|^{3/2}|d|^{1/2}} g_ag_{b}g_cg_{d}\|^2_{L^2_\omega},
\end{align}
where $\Delta_N^\epsilon (a,b,c,d)
$ and ${\mathcal A}_N(4)$ are defined in \eqref{Delta} and \eqref{giglMIT}.
Next (following Section~\ref{orth}) we split ${\mathcal A}_N(4)=\tilde {\mathcal A}_N(4)\cup \tilde {\mathcal A}_N^{c}(4)$,
where:
$$\tilde {\mathcal A}_N(4)=
\{(a,b,c,d)\in {\mathcal A}_N(4)| a\neq -b,
a\neq -c,a\neq-d, b\neq -c, b\neq -d, c\neq -d \}$$
and
\begin{multline*}
\tilde {\mathcal A}_N^{c}(4)
=\{(a,a,-a,-a)||a|\neq 0\}
\\
\cup \{(a, b, -a, -b), |a|\neq |b|\}
\\
\cup \{(a,b,-b,-a), |a|\neq |b|\}
\big),
\end{multline*}
(notice that we have excluded the vectors $(a,-a,b,-b)$ that in 
principle belong to ${\mathcal A}_N(4)$, but due to the cut-off property of $\psi_{\epsilon}$
they give a trivial contribution in the sum in \eqref{sumps}).
When we consider the sum on 
$\tilde {\mathcal A}_N(4)$, and we recall that $\Delta_N^\epsilon (a,b,c,d)=0$ 
when $|c+d|\leq N(1-\epsilon)$, then by orthogonality and Proposition~\ref{orth1}
we are reduced to:
$$\sum_{\substack{0<|a|, |b|, |c|, |d|\leq N\\|a+b|=|c+d|>N(1-\epsilon)}} \frac{1}{|a||b||c|^3|d|}
\leq O(\ln^2 N) \sum_{\substack{0<|c|, |d|\leq N(1-\epsilon)\\|c+d|>N(1-\epsilon)}} \frac{1}{|c|^3|d|}$$$$+ O(\ln^2 N)\sum_{\substack{0<|c|, |d|\leq N\\|c|>N(1-\epsilon)\\
|c+d|>N(1-\epsilon)}} \frac{1}{|c|^3|d|}+\sum_{\substack{0<|c|, |d|\leq N\\|d|>N(1-\epsilon)\\
|a+b|=|c+d|>N(1-\epsilon)}} \frac{1}{|a||b||c|^3|d|}.$$
The first and second terms on the r.h.s. are $O(\frac{\ln^3 N}{N})$
(in particular the first term is estimated by Lemma 3.3 in \cite{TVimrn}).
Concerning the third sum we can estimate it by
$$\sum_{n>N(1-\epsilon)} \big (\sum_{\substack{0<|a|, |b|<N\\|a+b|=n}} 
\frac 1{|a||b|}\big ) \cdot \big (\sum_{\substack{0<|c|\leq N\\ N(1-\epsilon)< |d|\leq 
N\\|c+d|=n}} 
\frac 1{|c|^3|d|}\big ), 
$$
and since by elementary considerations
$$\sum_{\substack{0<|a|, |b|<N\\|a+b|=n}} 
\frac 1{|a||b|} \lesssim \frac Nn$$,
we can continue the estimate above:
\begin{align*}...\lesssim \sum_{n>N(1-\epsilon)} \sum_{\substack{0<|c|\leq N\\ N(1-\epsilon)< |d|\leq 
N\\|c+d|=n}} 
\frac 1{|c|^3|d|}\lesssim \sum_{\substack{0<|c|\leq N\\ N(1-\epsilon)< |d|\leq 
N\\|c+d|>N(1-\epsilon)}} \frac 1{|c|^3|d|}\\\nonumber
\lesssim \sum_{\substack{0<|c|\leq N\\ N(1-\epsilon)< |d|\leq 
N}} \frac 1{|c|^3|d|}= O(\epsilon).
\end{align*}
Concerning the sum on $\tilde {\mathcal A}_N^{c}(4)$ we can apply the Minkowski inequality
and we get:
$$
\sum_{2|a|>N(1-\epsilon)} \frac 1{|a|^3} + \sum_{\substack{0<|a|, |b|\leq N\\|a+b|>N(1-\epsilon)}} \frac{1}{|a||b|^2}=
O(\frac{\ln N}{N})+ O(\epsilon),
$$
where the estimate of the second term is obtained as follows:
\begin{align}\label{adtil}\sum_{\substack{0<|a|, |b|\leq N\\|a+b|>N(1-\epsilon)}} \frac{1}{|a||b|^2}
\leq \sum_{\substack{0<|a|, |b|\leq N(1-\epsilon)\\|a+b|>N(1-\epsilon)}} \frac{1}{|a||b|^2}
+\sum_{\substack{0<|a|, |b|\leq N\\|b|>N(1-\epsilon)\\|a+b|>N(1-\epsilon)}} \frac{1}{|a||b|^2}
\\\nonumber+\sum_{\substack{0<|b|\leq N\\N(1-\epsilon)<|a|\leq N\\|a+b|>N(1-\epsilon)}} \frac{1}{|a||b|^2}
=O(\frac{\ln N}{N})+O(\epsilon).\end{align}
Here we have used Lemma 3.3 in \cite{TVimrn} to estimate the first term on the r.h.s.
The proof of \eqref{323sla} is similar.

\hfill$\Box$

We shall also need the following estimates.
\begin{prop}\label{fifth order} We have the following estimates:
\begin{align*}
&\|\int (S\varphi)^2 ({\mathcal H} S\varphi_x) [(Id - S^2) (S\varphi S\varphi_x)]\|_{L^2(d\mu_{3/2})}= O(\sqrt{\frac{\ln N}{N}})+O(\sqrt \epsilon),
\\\nonumber
&\|\int (S\varphi)^3 
{\mathcal H}[(Id - S^2) (S\varphi S\varphi_x)_x]\|_{L^2(d\mu_{3/2})}
= O(\sqrt{\frac{\ln N}{N}})+O(\sqrt \epsilon),
  \\\nonumber
& \| \int (S\varphi){\mathcal H}(S\varphi S \varphi_x) [(Id - S^2) (S\varphi S\varphi_x) ]\|_{L^2(d\mu_{3/2})}
= O(\sqrt{\frac{\ln N}{N}})+O(\sqrt \epsilon),
\\\nonumber
&\|\int (S\varphi)^2{\mathcal H}[S\varphi_x (Id - S^2) (S\varphi S\varphi_x)] \|_{L^2(d\mu_{3/2})}= O(\sqrt{\frac{\ln N}{N}})+O(\sqrt \epsilon)
\\\nonumber
&\|\int (S\varphi)^2{\mathcal H}[S\varphi (Id - S^2) (S\varphi S\varphi_x)_x]\|_{L^2(d\mu_{3/2})}= O(\sqrt{\frac{\ln N}{N}})+O(\sqrt \epsilon).
\end{align*}
\end{prop}

\begin{proof} We focus on the first estimate (the others can be treated by the same arguments as below). We have to prove 
\begin{align}\label{chloe5}&\|\sum_{(a,b,c,d,e)\in {\mathcal A}_N(5)}
\frac{sign(e)\Gamma^\epsilon_N(a,b,c,d,e) }{|a|^{3/2}|b|^{3/2}|c|^{1/2}|d|^{3/2}|e|^{1/2}}
g_ag_bg_cg_dg_e
\|_{L^2_\omega}\\\nonumber&=O(\sqrt{\frac{\ln N}{N}})+O(\sqrt \epsilon),
\end{align}
where
$\Gamma^\epsilon_N(a,b,c,d,e)$ and 
${\mathcal A}_N(5)$ are defined in \eqref{Gamma} and \eqref{Antilde}.
Next we split ${\mathcal A}_N(5)=\tilde {\mathcal A}_N(5)\cup \tilde {\mathcal A}_N^{c}(5)$ (see \eqref{splitprim})
and we get by Proposition \ref{orth1}
\begin{align}\label{impfan}&\|\sum_{(a,b,c,d,e)\in \tilde {\mathcal A}_N(5)}
\frac{sign(e) \Gamma^\epsilon_N(a,b,c,d,e) }{|a|^{3/2}|b|^{3/2}|c|^{1/2}|d|^{3/2}|e|^{1/2}}
g_ag_bg_cg_dg_e
\|_{L^2_\omega}^2
\\\nonumber
& \leq \sum_{\substack{0<|a|,|b|, |c|, |d| , |e|\leq N
\\|d+e|>N(1-\epsilon)}}
\frac{1}{|a|^3|b|^3|c||d|^3|e|}
\\\nonumber &\lesssim  \big( \sum_{\substack{|a| , |b|, |c|\leq N
\\|a+b+c|>N(1-\epsilon)}}
\frac{1}{|a|^3|b|^3 |c|}\big) \cdot \big( \sum_{\substack{|d| , |e|\leq N
\\|d+e|>N(1-\epsilon)}}
\frac{1}{|d|^3|e|}\big)
\end{align}
(here the constraint $|d+e|>N(1-\epsilon)$ comes from the cut--off $\psi_\epsilon$).
Next notice that
\begin{align}\label{casaver}&\sum_{\substack{|d| , |e|\leq N
\\|d+e|>N(1-\epsilon)}}
\frac{1}{|d|^3|e|}\\\nonumber
&\leq \sum_{\substack{0<|d| , |e|\leq N(1-\epsilon)
\\|d+e|>N(1-\epsilon)}}
\frac{1}{|d|^3|e|} +  \sum_{\substack{|d| \geq N(1-\epsilon)
\\|d+e|>N(1-\epsilon)}}
\frac{1}{|d|^3|e|}+ \sum_{\substack{N(1-\epsilon)\leq |e| \leq N
\\|d+e|>N(1-\epsilon)}}
\frac{1}{|d|^3|e|}\\\nonumber&=O(\frac{\ln N}{N}) + O(\epsilon)\end{align}
where we have used Lemma 3.3 of \cite{TVimrn} to estimate the first term
on the r.h.s.
By a similar argument
$$ \sum_{\substack{|a| , |b|, |c|\leq N
\\|a+b+c|>N(1-\epsilon)}}
\frac{1}{|a|^3|b|^3 |c|}=O(\frac{\ln N}{N}) + O(\epsilon).$$ 
Hence  the l.h.s. in \eqref{impfan} can be estimated by 
$O(\frac{\ln^2 N}{N^2}) + O(\epsilon^2)$.
Next we consider the set
\begin{equation}\label{Aortult}
 \tilde {\mathcal A}^{c}_{N,a=-b}= \{(a,b,c,d,e)\in {\mathcal A}_N(5)|a=-b\}.
 \end{equation}
By Proposition \ref{orth2} we can estimate the sum on the set 
$\tilde {\mathcal A}^{c}_{N,a=-b}\cap \{(a,b,c,d,e)|
|a+b+c|>N(1-\epsilon)\}$ by
$$\sum_{0<|a|\leq N} \frac 1{|a|^3} \big( \sum_{\substack{|d+e|>N(1-\epsilon), \\0<|d|,|e|\leq N, N(1-\epsilon)<|c|\leq N}} \frac{1}{|c||d|^3|e|}\big )^\frac 12\lesssim \sqrt \epsilon \big( \sum_{\substack{|d+e|>N(1-\epsilon), \\0<|d|,|e|\leq N}} \frac{1}{|d|^3|e|}\big )^\frac 12
$$
(the condition $|c|>N(1-\epsilon)$ comes from $a+b+c=c$ on the set $a=-b$),
and by \eqref{casaver} we can continue the estimate 
$$...\leq O(\sqrt \epsilon)[O(\sqrt {\frac{\ln N}{N}}) +O(\sqrt {\epsilon})].$$
We can treat the sum on $\tilde {\mathcal A}^{c}_{N,a=-c}$
by the following estimates:
$$\sum_{0<|a|\leq N} \frac 1{|a|^2} \big( \sum_{\substack{|d+e|>N(1-\epsilon), \\0<|d|,|e|\leq N}} \frac{1}{|b|^3|d|^3|e|}\big )^\frac 12\lesssim  \big( \sum_{\substack{|d+e|>N(1-\epsilon), \\0<|d|,|e|\leq N}} \frac{1}{|d|^3|e|}\big )^\frac 12
$$$$=O(\sqrt {\frac{\ln N}{N}}) +O(\sqrt {\epsilon}),$$
where we used again \eqref{casaver}.
Concerning the sum on $\tilde {\mathcal A}^{c}_{N,a=-d}$, 
we are reduced to the estimate:
\begin{align}\label{butcolnew} &\sum_{0<|a|\leq N} \frac 1{|a|^3} 
\big (\sum_{\substack{|a+b+c|>N(1-\epsilon)\\
0<|d|\leq N}} \frac 1{|b|^3|c||e|}\big )^{\frac 12}\\\nonumber \leq 
\sum_{0<|a|\leq N} \frac 1{|a|} 
[\big (\sum_{\substack{|a+b+c|>N(1-\epsilon)\\
|e+d|>N(1-\epsilon)\\0<|e|,|d|\leq N(1-\epsilon)}} &\frac 1{|b|^3|c||d|^4|e|}\big )^{\frac 12}
+\big (\sum_{\substack{|a+b+c|>N(1-\epsilon)\\
|e+d|>N(1-\epsilon)\\|d|\geq N(1-\epsilon)}} \frac 1{|b|^3|c||d|^4|e|}\big )^{1/2}]
\\
\nonumber &
+\sum_{0<|a|\leq N} \frac 1{|a|} \big (\sum_{\substack{|a+b+c|>N(1-\epsilon)\\
|e+d|>N(1-\epsilon)\\N(1-\epsilon)\leq |e|\leq N}} \frac 1{|b|^3|c||d|^4|e|}\big )^{1/2}.
\end{align}
Notice that the first two terms on the r.h.s. give a contribution 
$O(\frac{\ln^2 N}{\sqrt N})$ (in particular to estimate the first term we used
Lemma 3.3 in \cite{TVimrn}).
Concerning the last term, due to the fact $a=-d$ and $|-d+b+c|>N(1-\epsilon)$,
it can be controlled by
$$\sum_{0<|a|\leq N} \frac 1{|a|^2} \big (\sum_{\substack{|a+b+c|>N(1-\epsilon)\\
|e+d|>N(1-\epsilon)\\N(1-\epsilon)$$$$\leq |e|\leq N}} \frac 1{|b|^3|c||d|^2|e|}\big )^{1/2}
$$$$\leq \sum_{0<|a|\leq N} \frac 1{|a|^2} \sqrt \epsilon \big (\sum_{\substack{|-d+b+c|>N(1-\epsilon)}} \frac 1{|b|^3|c||d|^2}\big )^{1/2} =O(\sqrt \epsilon)[O(\sqrt {\frac{\ln N}N})+O(\sqrt \epsilon)].$$
At the last step we have used the estimate 
$$
\sum_{\substack{|-d+b+c|>N(1-\epsilon)}} \frac 1{|b|^3|c||d|^2}=O(\epsilon)+ O(\frac{\ln N}N),
$$
that in turn follows by splitting the constraint in four subdomains given by:
$$\{|-d+b+c|>N(1-\epsilon), |d|,|b|, |c|\leq N(1-\epsilon)\},
\{|-d+b+c|>N(1-\epsilon), |d| > N(1-\epsilon)\},$$
$$\{|-d+b+c|>N(1-\epsilon), |b| > N(1-\epsilon)\},
\{|-d+b+c|>N(1-\epsilon), |c| > N(1-\epsilon)\}.$$
The sum on the first constraint can be estimated by Lemma 3.2 in \cite{TVjmpa},
the estimate on the second and third one are trivial, and the last one gives a contribution $O(\epsilon)$.\\
Concerning the sum on $\tilde {\mathcal A}^{c}_{N,a=-e}$, notice that
the constraint $|d+e|>N(1-\epsilon)$ is equivalent to $|d-a|>N(1-\epsilon)$ and
hence we are reduced to treat:
\begin{align}\label{butcolnew345} &\sum_{0<|a|\leq N} \frac 1{|a|^2} 
\big (\sum_{\substack{|d-a|>N(1-\epsilon)\\
0<|b|, |c|, |d|\leq N}} \frac 1{|b|^3|c||d|^3}\big )^{\frac 12}\\\nonumber
\leq \sqrt {\ln N}
\big (\sum_{|a|\geq N(1-\epsilon)/2} \frac 1{|a|^2} 
&+ 
\sum_{|a|\leq N(1-\epsilon)/2} \frac 1{|a|^2} (\sum_{|d|\geq N(1-\epsilon)/2} 
\frac 1{|d|^3})^\frac 12\big )=O(\frac{\sqrt {\ln N}}N).
\end{align}
On the set $\tilde {\mathcal A}^{c}_{N,b=-c}$ 
the constraint $|a+b+c|>N(1-\epsilon)$ becomes $|a|>N(1-\epsilon)$
and hence the estimate follows by
$$\sum_{0<|b|\leq N} \frac 1{|b|^2} 
\big (\sum_{\substack{|a|>N(1-\epsilon)\\
0<|d|, |e|\leq N}} \frac 1{|a|^3|d|^3|e|}\big )^{\frac 12}=O(\frac{\sqrt {\ln N}}N).
$$
Notice that the sum on the sets 
$\tilde {\mathcal A}^{c}_{N,b=-d}$ and $\tilde {\mathcal A}^{c}_{N,b=-e}$
are similar to the sum on $\tilde {\mathcal A}^{c}_{N,a=-d}$ and $\tilde {\mathcal A}^{c}_{N,a=-e}$, that have been already treated.
Next we focus on the sum on the set  $\tilde {\mathcal A}^{c}_{N,c=-d}$.
In this case the constraint $|a+b+c|>N(1-\epsilon)$ becomes
$|a+b-d|>N(1-\epsilon)$ and we
are reduced to estimate
\begin{align*}&\sum_{0<|d|\leq N} \frac 1{|d|^2} 
\big (\sum_{\substack{|a+b-d|>N(1-\epsilon)\\
0<|d|, |e|\leq N}} \frac 1{|a|^3|b|^3|e|}\big )^{\frac 12}
\\\nonumber
&\leq \sqrt {\ln N}
\big (\sum_{|d|\geq N(1-\epsilon)/2} \frac 1{|d|^2} 
+ 
\sum_{|a+b|\geq N(1-\epsilon)/2} \frac 1{|a|^3 |b|^3}\big )^\frac 12=O(\frac{\sqrt {\ln N}}N).\end{align*}
The last sum to be considered is on the set  $\tilde {\mathcal A}^{c}_{N,c=-e}$,
where the constraint 
$|d+e|>N(1-\epsilon)$ can be rewritten as
$|d-c|>N(1-\epsilon)$,
hence we conclude by the following estimates:
\begin{align*}&\sum_{0<|c|\leq N} \frac 1{|c|} 
\big (\sum_{\substack{
0<|a|, |b|, |c|, |d|\leq N\\|d-c|>N(1-\epsilon)}} \frac 1{|a|^3|b|^3|d|^3}\big )^{\frac 12}
\leq \sum_{0<|c| \leq N(1-\epsilon)}
\big (\sum_{\substack{
0<|d| \leq N\\|d-c|>N(1-\epsilon)}} \frac 1{|c|^2|d|^3}\big )^{\frac 12}
\\\nonumber
& + \sum_{N(1-\epsilon)\leq |c| \leq N} \frac 1{|c|}
\leq \sum_{0<|c| \leq N(1-\epsilon)}
\big (\sum_{\substack{
0<|d| \leq N(1-\epsilon)\\|d-c|>N(1-\epsilon)}} \frac 1{|c|^2|d|^3}\big )^{\frac 12} 
\\\nonumber
&+ \sum_{
0<|c| \leq N(1-\epsilon)}
\big (\sum_{\substack{
|d| \geq N(1-\epsilon)\\|d-c|>N(1-\epsilon)}} \frac 1{|c|^2|d|^3}\big )^{\frac 12}+O(\epsilon)
=O(\frac 1{\sqrt N}) + O(\epsilon),\end{align*}
where we used Lemma 3.3 in \cite{TVjmpa} to estimate the first sum on the r.h.s.

\end{proof}

\begin{prop}\label{sixth}
We have the following estimates:
\begin{align*}
\|\int (S\varphi)^4 (Id - S^2) (S\varphi S\varphi_x)
\|_{L^2(d\mu_{3/2})}=O(\frac 1{\sqrt N}).
\end{align*}
\end{prop}
\begin{proof}
We have to show 
\begin{align}\label{andreanah}\|\sum_{(a,b,c,d,e,f)\in {\mathcal A}_N(6)}
\frac{sign(f)\Lambda^\epsilon_N(a,b,c,d,e,f) g_ag_bg_cg_dg_e g_f}{|a|^{3/2}|b|^{3/2}|c|^{3/2}|d|^{3/2}|e|^{3/2}
|f|^{1/2}}&\|_{L^2_\omega}= O(\frac 1{\sqrt N}), \end{align}
where:
$$\Lambda^\epsilon_N(a,b,c,d,e,f)=\psi(\frac aN)_\epsilon
\psi_\epsilon(\frac bN)\psi_\epsilon(\frac cN)\psi_\epsilon(\frac dN)\psi_\epsilon(\frac eN)\psi_\epsilon(\frac fN)[1- \psi_\epsilon^2(\frac{|e+f|}{N})]$$
and
\begin{equation*}{\mathcal A}_N(6)=\{(a,b,c,d,e,f)\in {\Z}^6|0<|a|, |b|, |c|, |d|,
|e|, |f|
\leq N, a+b+c+d+e+f=0\}.\end{equation*}
By the Minkowski inequality and by the cut--off property of $\psi_\epsilon$ 
we can estimate the l.h.s. in \eqref{andreanah} by 
\begin{align}\sum_{\substack{(a,b,c,d,e,f)\in {\mathcal A}_N(6)\\|a+b+c+d|>N(1-\epsilon)} }
\frac{1}{|a|^{3/2}|b|^{3/2}|c|^{3/2}|d|^{3/2}|e|^{3/2}
|f|^{1/2}}\\\nonumber
\leq \sum_{\substack{0<|a|,|b|,|c|,|d|,|e|\leq N\\|a+b+c+d|>N(1-\epsilon)} }
\frac{1}{|a|^{3/2}|b|^{3/2}|c|^{3/2}|d|^{3/2}|e|^{3/2}}=O(\frac 1{\sqrt N})
\end{align}
where we used
$|a+b+c+d|>N(1-\epsilon)$ implies $\max\{|a|, |b|, |c|, |d|\}\geq N/4(1-\epsilon)$.
\end{proof}

\noindent {\em Proof of Proposition \ref{leYu}.} For simplicity we denote $\tilde G=G_N^\epsilon$ and $S=S^\epsilon_N$, then
we get for a general function $v$ 
\begin{align}\label{ide}&\tilde G(v)= E_{3/2}(Sv) + (\|v\|_{\dot H^{3/2}}^2 - \|Sv\|_{\dot H^{3/2}}^2),
\end{align}
where
\begin{align}\label{Epri}E_{3/2}(u)= \|u\|^2_{\dot H^{3/2}} - (\int \frac 32 u u_x^2 +\frac 12 u (Hu_x)^2) \\\nonumber
- \int (\frac 13 u^3 Hu_x +\frac 14 u^2H(uu_x)) - \frac1{20} \int u^5.\end{align}
Arguing as in Proposition~\ref{leY} and by using the notation $u(t)=\pi_N\Phi_N^\epsilon(t) \varphi$ we get:
\begin{align} \label{32I}&\frac d{dt} E_{3/2}(Su)_{t=0}=
[\int Su_t {\mathcal H}S\varphi_{xxx} + \int S\varphi {\mathcal H} Su_{txxx} \\\nonumber -(\int \frac 32 & S u_t (S\varphi_x)^2 +\frac 12 S u_t ({\mathcal H}S\varphi_x)^2) - (\int 3 S \varphi 
S\varphi_xSu_{tx} + S \varphi ({\mathcal H}S\varphi_x) {\mathcal H}Su_{tx})\\\nonumber
&- \int ( S \varphi)^2 Su_t {\mathcal H}S \varphi_x - \int \frac 13 (S \varphi)^3 {\mathcal H}S u_{tx} \\
\nonumber & - \frac 12 \int (S\varphi) S u_t {\mathcal H}(S\varphi S\varphi_x) 
- \frac 14 \int (S\varphi)^2{\mathcal H}(Su_t S\varphi_x) \\
\nonumber &- \frac 14 \int (S\varphi)^2{\mathcal H}(S\varphi Su_{tx}) - \frac1{4} \int (S\varphi)^4 Su_t ]_{S u_t= (Id - S^2) S\varphi S\varphi_x}.
\end{align}
Moreover by using the equation solved by $u(t,x)$ we get:
\begin{align}\nonumber
&\frac{d}{dt} (\|u\|_{\dot H^{3/2}}^2 - \|Su\|_{\dot H^{3/2}}^2)_{t=0} \\
\nonumber &=\int  (- {\mathcal H} \varphi_{xx} - S(S\varphi S\varphi_x)) {\mathcal H}\varphi_{xxx}  
+  \int  \varphi {\mathcal H}(- {\mathcal H} \varphi_{xx} - S(S\varphi S \varphi_x))_{xxx} \\
\nonumber & +\int  ( {\mathcal H} S\varphi_{xx} + S^2(S\varphi S\varphi_x)) {\mathcal H}S \varphi_{xxx}  
+  \int  S \varphi {\mathcal H}({\mathcal H} S\varphi_{xx}  + S^2(S\varphi S\varphi_x))_{xxx}.
\end{align}
By properties of ${\mathcal H}$ and integration by parts we get:
\begin{align}\label{32II}
&\frac{d}{dt} (\|u\|_{\dot H^{3/2}}^2 - \|Su\|_{\dot H^{3/2}}^2)_{t=0} \\
\nonumber &=\int  (- S(S\varphi S\varphi_x)) {\mathcal H}\varphi_{xxx}  
+  \int  \varphi {\mathcal H}(- S(S\varphi S \varphi_x))_{xxx} \\
\nonumber & +\int  (S^2(S\varphi S\varphi_x)) {\mathcal H}S \varphi_{xxx}  
+  \int  S \varphi {\mathcal H}(S^2(S\varphi S\varphi_x))_{xxx}\\\nonumber 
&= 2 \int  \varphi {\mathcal H}(- S(S\varphi S \varphi_x))_{xxx} +2 \int  (S^2(S\varphi S\varphi_x)) {\mathcal H}S \varphi_{xxx}. 
\end{align}
By combining \eqref{32I}, \eqref{32II} with \eqref{ide} we deduce that the terms of order three
in the expression $\frac d{dt} (\tilde G (\pi_N \Phi_N^\epsilon(t) \varphi)_{t=0}$  give a trivial  contribution, as the following computation shows:
\begin{align*}&\hbox{ cubic terms in \eqref{32I}+ \eqref{32II} }
\\\nonumber &=2 \int  \varphi {\mathcal H}(- S(S\varphi S \varphi_x))_{xxx} +2 \int  (S^2(S\varphi S\varphi_x)) {\mathcal H} S \varphi_{xxx}  
\\\nonumber &+\int (Id - S^2) (S\varphi S\varphi_x) {\mathcal H}S\varphi_{xxx} + \int S\varphi {\mathcal H} (Id - S^2) (S\varphi S\varphi_x)_{xxx}\\\nonumber
&= -2 \int  \varphi {\mathcal H} (S(S\varphi S \varphi_x))_{xxx} +2 \int  (S^2(S\varphi S\varphi_x)) {\mathcal H}S \varphi_{xxx}  
\\\nonumber & +\int (S\varphi S\varphi_x) {\mathcal H}S\varphi_{xxx} + \int S\varphi {\mathcal H} (S\varphi S\varphi_x)_{xxx}
\\\nonumber & -\int S^2 (S\varphi S\varphi_x) {\mathcal H}S\varphi_{xxx} - \int S\varphi {\mathcal H} S^2 (S\varphi S\varphi_x)_{xxx}=0.
\end{align*}
Next by combining again \eqref{32I}, \eqref{32II} with \eqref{ide} we compute the contribution to $\frac d{dt} (\tilde G (\pi_N \Phi_N^\epsilon(t) \varphi)_{t=0}$
given by terms of order four:
\begin{align*}&\hbox{ quartic terms in \eqref{32I}+ \eqref{32II} }
\\\nonumber &=
- \int \frac 32  (Id - S^2) (S\varphi S\varphi_x) (S\varphi _x)^2 - 3 S \varphi S\varphi_x
(Id - S^2) (S\varphi S\varphi_x)_x
\\\nonumber &- \int \frac 12 (Id - S^2) (S\varphi S\varphi_x) ({\mathcal H}S\varphi_x)^2 -
S \varphi ({\mathcal H} S\varphi_x) {\mathcal H} (Id - S^2) (S\varphi S\varphi_x)_x\\\nonumber
&=
- \int \frac 32  (Id - S^2) (S\varphi S\varphi_x) (S\varphi _x)^2 
- \int \frac 12 (Id - S^2) (S\varphi S\varphi_x) ({\mathcal H} S\varphi_x)^2 \\\nonumber &-
S \varphi ({\mathcal H} S\varphi_x) {\mathcal H} (Id - S^2) (S\varphi S\varphi_x)_x,
\end{align*}
where we used
$$2 \int  S \varphi S\varphi_x
(Id - S^2) (S\varphi S\varphi_x)_x
= \int \partial_x ((Id - S^2)^{1/2} (S\varphi S\varphi_x))^2 =0.
$$
By Proposition~\ref{spca2} we can estimate the terms above.
Next we focus on the quintic terms:
\begin{align*} 
\hbox{ quintic terms in \eqref{32I}+ \eqref{32II} }
\\\nonumber =- \int (S\varphi)^2 ({\mathcal H} S\varphi_x) [(Id - S^2) (S\varphi S\varphi_x)]
& - \frac 13 \int (S\varphi)^3 
{\mathcal H}[(Id - S^2) (S\varphi S\varphi_x)_x]  \\\nonumber
 - \frac 12 \int (S\varphi){\mathcal H}(S\varphi S \varphi_x) &[(Id - S^2) (S\varphi S\varphi_x) ]
\\\nonumber - \frac 14 \int (S\varphi)^2{\mathcal H}[S\varphi_x (Id - S^2) (S\varphi S\varphi_x)]
&- \frac 14 \int (S\varphi)^2{\mathcal H}[S\varphi (Id - S^2) (S\varphi S\varphi_x)_x]
\end{align*}
Notice that those terms can be controlled by Proposition \ref{fifth order}.
Next we notice that 
\begin{align*}&\hbox{ sixth order terms in \eqref{32I}+ \eqref{32II} }
\\\nonumber =
& 
- \frac1{4} \int (S\varphi)^4 (Id - S^2) (S\varphi S\varphi_x)
\end{align*}
and they can be controlled thanks to Proposition \ref{sixth}.

\hfill$\Box$


\end{document}